\documentclass[12pt,reqno]{amsart}

\usepackage{hyperref}
\usepackage[usenames]{color}
\usepackage[utf8]{inputenc} 
\usepackage{float}
\usepackage{bbm}

\usepackage[displaymath, mathlines]{lineno}

\usepackage[english]{babel} 

\usepackage{amssymb, amsmath}
\usepackage{geometry,graphicx}

\usepackage{cite}
\usepackage{enumitem}

\usepackage{soul}

\usepackage{xurl}

\allowdisplaybreaks

\theoremstyle{plain}
\newtheorem{theorem}{Theorem}[section] 

\newtheorem{lemma}[theorem]{Lemma} 

\newtheorem{definition}[theorem]{Definition}

\newtheorem{corollary}[theorem]{Corollary}

\newtheorem{remark}[theorem]{Remark}

\newtheorem{example}[theorem]{Example}

\numberwithin{equation}{section} 

\begin{document}

\begin{center}
 \textbf{Characterizations of fractional operators via integral transforms}
\end{center}

\begin{center}
 Daniel Cao Labora $^a$, Marc Jornet $^b$
\end{center}

\begin{center}
$^a$ Department of Statistics, Mathematical Analysis and Optimization, Faculty of Mathematics and CITMAga, Universidade de Santiago de Compostela (USC), Rua Lope Gomez de Marzoa, 6, 15782, Santiago de Compostela, Galicia, Spain. \\
email: daniel.cao@usc.es \\ 
ORCID: 0000-0003-2266-2075
\end{center}

\begin{center}
$^b$ Departament de Matem\`atiques, Universitat Jaume I, 12071 Castell\'o, Spain. \\
email: jornet@uji.es \\ 
ORCID: 0000-0003-0748-3730
\end{center}

\ \\
\textbf{Abstract.} In 1974, J. S. Lew established a reasonable conjecture regarding an axiomatic characterization for the one-dimensional Riemann--Liouville integral. This conjecture was proved by Cartwright and McMullen in 1978. After that, little further work has been done on this topic, except some extensions for the Stieltjes case in one and several variables. In this paper, we prove the necessity of the axioms established in the conjecture of J. S. Lew using the Cauchy functional equation and Hamel bases. In addition, we give a proof for the characterization in several variables by employing the Titchmarsh theorem, as a natural extension of the approach of Cartwright and McMullen. We also provide an alternative version and proof in one and several variables with Laplace transforms and the Cauchy functional equation, weakening parts of the continuity assumption. We show a similar result for the Riesz potential in terms of the Fourier transform. Finally, we illustrate how the theory can be used for characterization in the context of fractional calculus with respect to a non-smooth integrator, based on transmutation and measures. 
\\
\textit{AMS Classification 2020: 26A33, 26B99, 44A10.}

\tableofcontents

\section{Introduction}

The Riemann--Liouville integral is a very important operator in mathematical analysis and, in particular, in the field of fractional calculus, which consists of the study of derivatives, integrals and differential equations of non-integer order~\cite{diethelm}. Solving a conjecture posed by Lew in 1974, Cartwright and McMullen proved in 1978 that the one-dimensional Riemann--Liouville integral is essentially the unique valid extension of the ordinary integral in a fractional sense, considering its semigroup property and continuity with respect to the parameter~\cite{cart}. Recently, in 2025, Cao Labora developed an extension of the Cartwright--McMullen theorem for the Stieltjes case with the concept of transmutation (algebraic conjugation) in fractional calculus~\cite{cao_jie}, based on his PhD thesis~\cite{tese}. The multidimensional extension of this article with a weight was given by Jornet in~\cite{meu}, where open problems were described at the end: in particular, it was unknown whether the fourth condition in the theorems of~\cite{meu}, concerning dimensionality reduction, is strictly necessary.

\subsection{Summary}

The necessity of the axioms given by Cartwright and McMullen has not been discussed in the literature, although it is an interesting topic. As shown by Cao Labora \cite{cao_jie,tese}, the interpolation and the index law in the real context already imply the uniqueness of the one-dimensional Riemann--Liouville integral for rational orders, and then the continuity with respect to the fractional parameter allows the conclusion. However, is continuity really necessary? This question is intimately connected with additivity together with discontinuity in the Cauchy functional equation; this area of research is classical and involves Zorn's lemma and Hamel bases of $\mathbb{R}$ over $\mathbb{Q}$, see \cite[Chapter~3]{saaty} and \cite{aczel}. In the second section of the document, we use these ideas to demonstrate that the properties given by Cartwright and McMullen are indeed necessary.

In the third section, we develop a proof for the natural extension of the Cartwright--McMullen theorem in several variables. This proof is inspired by the techniques presented in \cite{cart}.

In the fourth section, we provide an alternative proof of the Cartwright--McMullen theorem that avoids using the Titchmarsh theorem and weakens parts of the continuity assumption. The main idea of the proof is the following. The Laplace transform of the Riemann--Liouville integral can be viewed as a ``fractionalization'' of the Laplace transform of the standard integral. This observation motivates us to use the Laplace transform to establish a uniqueness result that trivially implies the Cartwright--McMullen theorem. Moreover, the Riemann--Liouville integral is related to the beta distribution, and therefore it is natural to characterize it via the moment-generating function, which is essentially the Laplace transform. The characterization is done in two steps, using the theory of Cauchy functional equations. We first assume that the operator is defined via a convolution with a certain kernel and, later, we are able to eliminate such a hypothesis. The fifth section is devoted to the extension of these results to dimensions greater than or equal to two. From the third to the fifth section of this paper, we solve the recent open problem raised in~\cite{meu}.

In the sixth section, we develop a similar procedure in order to obtain an axiomatic characterization for the Riesz potential, but by means of the Fourier transform. In this case, we are only able to provide a proof for the convolution case. 

In the seventh section, we characterize the Riemann--Liouville integral with respect to a strictly increasing continuous integrator, using transmutation. Compared with recent work \cite{cao_jie,meu}, we do not assume a continuously differentiable integrator, because we utilize tools from measure theory such as pushforward measures.

\subsection{Some previous results}

The following sections utilize techniques from the theory of functional equations, which is not typically covered in standard fractional calculus works. Consequently, let us introduce some basic notions and results.
Consider a function $h: \Omega \subseteq \mathbb{R}^n \to \mathbb{R}^n$, where $\Omega$ fulfills the property $x,y \in \Omega \implies x+y \in \Omega$. We say that $h$ satisfies the Cauchy functional equation if the following identity holds:
\begin{equation} \label{Cauchy}
    h(x+y)=h(x)+h(y), \quad \forall x,y \in \Omega.
\end{equation}

It is well known that when $h$ is assumed to be continuous and $\Omega = \mathbb{R}^n$, any solution to \eqref{Cauchy} is a linear function. In fact, it suffices to assume that $h$ is continuous at one point $x^0\in \Omega=\mathbb{R}^n$, because this fact together with~\eqref{Cauchy} implies global continuity. Indeed, since $h(x^0+\epsilon)=h(x^0)+h(\epsilon)\rightarrow h(x^0)$ when $\epsilon\rightarrow 0$, we obtain continuity at $0$, and then $h(y+\epsilon)=h(y)+h(\epsilon)\rightarrow h(y)$, that is, $h$ is continuous at any $y$.

Due to the nature of our problem, we will be interested in a similar result when $\Omega\subseteq (0,\infty)^n$. Therefore, we will take into account the following results from \cite{Kuczma}.

\begin{lemma}[Theorem 4.5.1 in \cite{Kuczma}]
Let $(H, +)$ be a commutative group, and $(S, +)$ a subsemigroup of $(H, +)$. Then the group $G$ generated by $S$ is $G = S - S$.
\end{lemma}

\begin{lemma}[Corollary 18.2.1 in \cite{Kuczma}] \label{lem_Ku} Let $(X, +)$ be a commutative group, let $(Y, +)$ be a group, and let $(S, +)$ be a subsemigroup of $(X, +)$ generating $X$. Let $g : S \to Y$ be a homomorphism. Then there exists a unique homomorphism $f : X \to Y$ such that $f_{|S} = g$.
\end{lemma}

In Lemma~\ref{lem_Ku}, the function $f$ can be defined as follows: Given $x\in S$, it can be written as $x=s_1-s_2$, for certain $s_1,s_2\in S$ (not necessarily unique), and one takes $f(x)=g(s_1)-g(s_2)$. The proof of the lemma follows from standard computations.

Since $(\mathbb{R}^n,+)$ is generated by $(\Omega,+)$ with $\Omega=\prod_{j=1}^n (a_j,\infty)\subseteq (0,\infty)^n$, we deduce the following corollary from the two previous results.

\begin{corollary} \label{linear} Any function $h: \Omega=\prod_{j=1}^n (a_j,\infty) \subseteq (0,\infty)^n \to \mathbb{R}^n$ or $h: \Omega=\prod_{j=1}^n [a_j,\infty) \subseteq (0,\infty)^n \to \mathbb{R}^n$ that is continuous at one point and satisfies equation~\eqref{Cauchy} has a unique extension to $\mathbb{R}^n$ that is continuous at one point and still satisfies equation \eqref{Cauchy}. In particular, the extension of $h$ is a linear function and, consequently, so is $h$.
\end{corollary}
\begin{proof}
In the notation of Lemma~\ref{lem_Ku}, we have that $(X, +)=(Y,+)=(\mathbb{R}^n,+)$ is a commutative group and $(S,+)=(\Omega,+)$ is a semigroup that generates $X$. Since $g=h$ satisfies equation~\eqref{Cauchy} on $\Omega$, it is a homomorphism. Then, by Lemma~\ref{lem_Ku}, there exists a unique homomorphism $f:\mathbb{R}^n\rightarrow\mathbb{R}^n$ such that $f_{|\Omega}=h$. As $f$ is a homomorphism, it satisfies equation~\eqref{Cauchy} on $\mathbb{R}^n$. Since $h$ is continuous at one point $x_0\in\Omega$, we deduce that $f$ is also continuous at $x_0$, because it is an extension. Thus, we simply apply the standard result of the Cauchy functional equation on $\mathbb{R}^n$, which ensures that $f$ is linear on $\mathbb{R}^n$, and in particular its restriction $h$ is linear on $\Omega$.
\end{proof}

Now we move to the Cauchy functional equation on a restricted domain. In particular, we will be interested in $\Omega=(0,n)\subseteq (0,\infty)$, where $n$ is a natural number. Consider a function $h:\Omega=(0,n)\rightarrow\mathbb{R}$. We say that $h$ satisfies the (restricted) Cauchy functional equation if 
\begin{equation} h(x+y)=h(x)+h(y),\quad \forall x,y\in\Omega\text{ such that }x+y\in\Omega. 
\label{cfd2}
\end{equation}

\begin{lemma} \label{cfr}
If $h:\Omega=(0,n)\rightarrow\mathbb{R}$ is a function that is continuous at one point and satisfies equation~\eqref{cfd2}, then $h$ is linear.
\end{lemma}
\begin{proof}
We extend $f$ to an additive function on $(0,\infty)$, so that Corollary~\ref{linear} can be applied.

Let $\eta\in [n,2n)$. We can write $\eta=\alpha+\beta$, where $0<\alpha<n$ and $0<\beta<n$. Then we define the extension $f(\eta):=f(\alpha)+f(\beta)$.

Let us see that this definition on $[n,2n)$ is well-posed. If $\eta$ were also $\tilde{\alpha}+\tilde{\beta}$, where $0<\tilde{\alpha}<n$ and $0<\tilde{\beta}<n$, then we consider without loss of generality that $\alpha<\tilde{\alpha}$ and $\tilde{\beta}<\beta$. By additivity on $(0,n)$, we know that $f(\tilde{\alpha})=f(\tilde{\alpha}-\alpha)+f(\alpha)$ and $f(\beta)=f(\beta-\tilde{\beta})+f(\tilde{\beta})$. Since $\beta-\tilde{\beta}=\tilde{\alpha}-\alpha$, this implies $f(\tilde{\alpha})+f(\tilde{\beta})=f(\alpha)+f(\beta)$, as desired.

Let us check that $f$ is additive on $[n,2n)$. Suppose that $\alpha\geq n$, $\beta<n$, and $\alpha+\beta<2n$. Since $0\leq \alpha-n<n-\beta$, we can pick a number $\alpha_2$ such that $\alpha-n<\alpha_2<n-\beta$, that is, $\alpha-\alpha_2<n$ and $\beta+\alpha_2<n$. Thus, by the definition of $f$, $f(\alpha+\beta)=f(\alpha-\alpha_2)+f(\beta+\alpha_2)$. On the one hand, $f(\beta+\alpha_2)=f(\beta)+f(\alpha_2)$, by additivity on $(0,n)$. On the other hand, again by the definition of $f$, $f(\alpha)=f(\alpha_2)+f(\alpha-\alpha_2)$. Therefore, we arrive at $f(\alpha+\beta)=f(\alpha)+f(\beta)$, as wanted.

Once we have the extension to $[n,2n)$, we can proceed analogously for $[2n,4n)$, $[4n,8n)$, etc.
\end{proof}

\section{Necessity of the axioms in the Cartwright--McMullen theorem} \label{sec_neces}

We write $I_a^\alpha$ for the one-dimensional Riemann--Liouville integral with left origin at $a\in\mathbb{R}$, for $\alpha\in (0,\infty)$:
\[ I_a^\alpha g(t)=\frac{1}{\Gamma(\alpha)}\int_a^t (t-\tau)^{\alpha-1}g(\tau)\mathrm{d}\tau. \]
In particular, $I^\alpha=I_0^\alpha$. In this section, we consider the following two versions of the Cartwright--McMullen theorem that characterize $I_a^\alpha$.

\begin{theorem}[Cartwright--McMullen, real version] \label{t1} \cite[Theorem~2.18]{tese}, \cite[Theorem~2.1]{cao_jie}
Given $a\in\mathbb{R}$, there is only one family of operators $(J_a^\alpha)_{\alpha>0}$ on $\mathrm{L}^1(a,T)$ that satisfies the following conditions:
\begin{itemize}
\item[1.] $J_a^1=I_a^1$;
\item[2.] $J_a^\alpha J_a^\beta=J_a^{\alpha+\beta}$ for all $\alpha,\beta>0$ (index-law property);
\item[3.] The family is continuous with respect to the parameter, $\alpha\mapsto J_a^\alpha$ from $(0,\infty)$ into the set of continuous linear maps $\mathcal{B}(\mathrm{L}^1(a,T))$.
\end{itemize}
We find that $J_a^\alpha=I_a^\alpha$ is the Riemann--Liouville integral on $\mathrm{L}^1(a,T)$, for all $\alpha\in (0,\infty)$.
\end{theorem}

\begin{theorem}[Cartwright--McMullen, complex version] \label{t2} \cite{cart}
Given $a\in\mathbb{R}$, there is only one family of operators $(J_a^\alpha)_{\alpha>0}$ on $\mathrm{L}^1((a,T),\mathbb{C})$ that satisfies the following conditions:
\begin{itemize}
\item[1$^\prime$.] $J_a^1=I_a^1$;
\item[2$^\prime$.] $J_a^\alpha J_a^\beta=J_a^{\alpha+\beta}$ for all $\alpha,\beta>0$ (index-law property);
\item[3$^\prime$.] The family is continuous with respect to the parameter, $\alpha\mapsto J_a^\alpha$ from $(0,\infty)$ into the set of continuous linear maps $\mathcal{B}(\mathrm{L}^1((a,T),\mathbb{C}))$;
\item[4$^\prime$.] If $y\geq0$, then $J_a^\alpha y\geq0$ for all $\alpha>0$.
\end{itemize}
We find that $J_a^\alpha=I_a^\alpha$ is the Riemann--Liouville integral on $\mathrm{L}^1((a,T),\mathbb{C})$, for all $\alpha\in (0,\infty)$.
\end{theorem}

The incorporation of assumption 4$^\prime$ in the complex scenario is based on certain roots of unity that appear in the proofs of the previous theorems. Condition~4$^\prime$ ensures that only the root $1$ is considered after applying the Titchmarsh theorem. In fact, as justified in~\cite{cao_jie}, the root $-1$ can be discarded due to the index law, hence 4$^\prime$ could be replaced by a weaker condition 4$^{\prime\prime}$ stating ``If $y([a,T])\subseteq \mathbb{R}$, then $J_a^\alpha y((a,T))\subseteq\mathbb{R}$ for all $\alpha>0$''.

Another remark is that with 1 and 2 or with 1$^\prime$, 2$^\prime$ and 4$^{\prime\prime}$, one can prove that $J_a^\alpha=I_a^\alpha$ for all rational orders $\alpha>0$. Later, with 3 or 3$^\prime$, the equality is extended to irrational orders $\alpha>0$. It is relevant to study whether 3 and 3$^\prime$ are necessary for this extension, because equality for rational orders is already a strong condition.

We study whether all properties 1--3 and 1$^\prime$--4$^\prime$ are actually necessary to characterize the Riemann--Liouville integral. 

\noindent
\textbf{Properties 1 and 3 hold, but 2 not:}

A simple example is $J_a^\alpha = \alpha I_a^1$.

\noindent
\textbf{Properties 2 and 3 hold, but 1 not:}

Take $J_a^\alpha= I_a^{2\alpha}$ or $J_a^\alpha=2^\alpha  I_a^{\alpha}$.

\noindent
\textbf{Properties 1$^\prime$, 2$^\prime$ and 3$^\prime$ hold, but 4$^\prime$ not:}

Consider $J_a^\alpha=\mathrm{e}^{2\pi \mathrm{i}\alpha}  I_a^{\alpha}$, where $\mathrm{i}$ is the imaginary unit.

\noindent
\textbf{Properties 1 and 2 hold, but 3 not:}

Note that the index rule is related to \textit{additivity}, while property~3 deals with \textit{continuity}. Therefore, we are looking for a function or operator that is additive but not continuous. This topic is classical in the context of the Cauchy functional equation. We adapt the ideas of Acz\'el and Erd\H{o}s from 1965~\cite{aczel}. The general theory on the one-dimensional Cauchy functional equation is elaborated in Saaty's book \cite[Chapter~3]{saaty}.

Let $\mathcal{H}=(b_\lambda)_{\lambda\in\Lambda}$ be a Hamel basis of $\mathbb{R}$ over $\mathbb{Q}$ \cite[page~127]{saaty}, where $\Lambda$ is a nondenumerable set of indices and two elements of $\mathcal{H}$ are $\pi$ and $\pi+1$ (note that in that case, no rational number can be an element of the basis; otherwise, there would be linear dependence). This basis is shown to exist by Zorn's lemma or, equivalently, transfinite induction. Every $\alpha>0$ is written uniquely as $\alpha=q_1b_{\lambda_1}+\ldots+q_s b_{\lambda_s}$, where $q_1,\ldots,q_s\in\mathbb{Q}$, $\lambda_1,\ldots,\lambda_s\in\Lambda$ and $s\in\mathbb{N}$ depend on $\alpha$. We notice that $b_\lambda$ or $q_j$ may be negative. We also notice that, if $\alpha=1$, then $1=1\cdot (\pi+1)+(-1)\cdot \pi$, that is, we have two $q_j$'s of values $1$ and $-1$, respectively; in particular, the sum of the rational coefficients is $0$.

With the representation of $\alpha$, let $J_a^\alpha=2^{q_1+\ldots+q_s} I_a^{\alpha}$. 

We have $J_a^1=2^{1-1} I_a^{1}=I_a^1$. This is property~1. In particular, for all $\gamma\in\mathbb{Q}$, we obtain $J_a^\gamma= I_a^\gamma$, because $\gamma=\gamma (\pi+1)+(-\gamma)\pi$. This fact is consistent with the proof of the real version of the Cartwright--McMullen theorem before using the third property \cite{tese,cao_jie}.

If $\alpha=q_1b_{\lambda_1}+\ldots+q_s b_{\lambda_s}$ and $\tilde{\alpha}=\tilde{q}_1 b_{\tilde{\lambda}_1}+\ldots+\tilde{q}_r b_{\tilde{\lambda}_r}$, then $J_a^{\alpha}\circ J_a^{\tilde{\alpha}}=2^{q_1+\ldots+q_s} I_a^{\alpha} \circ 2^{\tilde{q}_1+\ldots+\tilde{q}_r} I_a^{\tilde{\alpha}}=2^{q_1+\ldots+q_s+\tilde{q}_1+\ldots+\tilde{q}_r} I_a^{\alpha+\tilde{\alpha}}=J_a^{\alpha+\tilde{\alpha}}$. This is property~2.

Let us see that property~3 does not hold. Suppose that $J_a^\alpha$ is continuous with respect to $\alpha$. The real maps $\alpha\mapsto \|J_a^\alpha\|$ or $\alpha\mapsto \|J_a^\alpha 1\|$ are then continuous on $(0,\infty)$, where $\|\cdot\|$ is the norm. Observe that $\|J_a^{\alpha}\|=\|J_a^\alpha 1\|=2^{q_1+\ldots+q_s} (T-a)^{\alpha}/\Gamma(\alpha+1)$, so $\alpha\mapsto 2^{q_1+\ldots+q_s}$ is necessarily continuous. However, its image is countable; hence a disconnected set, and the map cannot be continuous.

With these results, we obtain the following theorem.

\begin{theorem} \label{th_main}
The assumptions of Theorem~\ref{t1} and Theorem~\ref{t2} are necessary. (The necessity of condition~3 or, equivalently, C, considers the validity of Zorn's lemma in the proof.)
\end{theorem}

Although the previous result shows that all hypotheses are necessary, it can be the case that if we restrict the candidates to a certain family of operators, some hypotheses are no longer necessary. In that sense, we have the following result.

\begin{theorem} \label{th222}
Given an arbitrary function $f:(0,\infty)\rightarrow [0,\infty)$, if the operators $J_a^\alpha= I_a^{f(\alpha)}$ satisfy properties~1 and~2, then $f$ is the identity map. Property~3 is not required. 
\end{theorem}

\begin{proof}
The proof of this fact proceeds as follows. By property~2, $J_a^{\alpha+\beta}= I_a^{f(\alpha+\beta)}$ is equal to $J_a^\alpha J_a^\beta= I_a^{f(\alpha)} I_a^{f(\beta)}= I_a^{f(\alpha)+f(\beta)}$, therefore $f(\alpha+\beta)=f(\alpha)+f(\beta)$ for all $\alpha,\beta>0$. Notice that $f$ is an increasing function, because $f(\alpha+\beta)\geq f(\alpha)$ (here the non-negative codomain of $f$ is key, compared to the standard Cauchy functional equation). On the other hand, since $f$ is an additive function, it is easy to check that $f(p)=cp$ for all positive numbers $p\in\mathbb{Q}$, where $c=f(1)$. We need to discuss what happens when $p$ is not a rational number. Pick a sequence of positive rational numbers $q_n$ that converges to $p$ such that $0<q_n\leq p\leq q_n+1/n$. Then $0\leq n(p-q_n)\leq 1$, and $0\leq f(n(p-q_n))=f(np)-f(nq_n)=nf(p)-ncq_n\leq f(1)$ (we use the fact that $f$ is increasing). Consequently, $0\leq f(p)-cq_n\leq f(1)/n$ and $f(p)=cp$ when $n\rightarrow\infty$. By property~1, $f(1)=1$, so $c=1$ and $f$ is the identity map.
\end{proof}

\begin{remark} \normalfont
We observe that, as a consequence of Theorem~\ref{th222}, we recover the nice result of Acz\'el and Erd\H{o}s, who state that there is no ``Hamel basis'' from $\mathbb{R}\cap (0,\infty)$ over $\mathbb{Q}\cap (0,\infty)$. This is written in quotation marks because these sets are not vector spaces, but essentially what we mean is that there is no Hamel basis for the positive numbers formed by positive real numbers with positive rational coefficients. If we had such a basis, then for $J_a^\alpha= I_a^{f(\alpha)}$, we could choose an additive discontinuous function $f:(0,\infty)\rightarrow [0,\infty)$ of the form $f(\alpha)=q_1+\ldots+q_s$ with $q_1,\ldots,q_s\in\mathbb{Q}\cap(0,\infty)$, so property~3 would be necessary to ensure that $f$ is the identity map. However, Theorem~\ref{th222} tells us that this conclusion is not true.
\end{remark}

\begin{remark} \normalfont
As Cao Labora and colleagues showed in 2018~\cite{dif}, the index rule cannot be true for a fractional derivative. Our approach gives the intuition of why this fact occurs, focusing on the Riemann--Liouville derivative. The Riemann--Liouville derivative can be defined by analytic continuation of the Riemann--Liouville integral, so that we can denote it by $I_a^{-\alpha}$, $\alpha>0$. If the Riemann--Liouville derivative has an index law, then in the proof of Theorem~\ref{th222} we can take $f:\mathbb{R}\rightarrow\mathbb{R}$. However, in such a case, $f$ is not increasing, and we are in the standard situation of the Cauchy functional equation, with an additive function $f$ that can be discontinuous by invoking Zorn's lemma. This $f$ is not the identity map; therefore, there is an alternative operator to the Riemann--Liouville derivative and integral of the form $J_a^\alpha= I_a^{f(\alpha)}$, which is not possible by Theorem~\ref{th222}.
\end{remark}

\section{The Cartwright--McMullen theorem in several variables} \label{sec3d}

In this section, we provide a direct proof for the Cartwright--McMullen theorem in several variables. The statement is a natural generalization of the one in \cite{cart}. We highlight that for some hypotheses, the boxes for multiorders are open and, for other ones, they are closed. This is not an error, but a deliberate choice. Similarly to the notation used in the literature, if $\alpha=(\alpha_1,\dots,\alpha_n) \in [0,\infty)^n$, the operator $I^\alpha$ will denote the Riemann--Liouville integral of multiorder $\alpha$ in dimension $n>1$ with origin $(0,\ldots,0)$; that is, we integrate $\alpha_i$ times with respect to coordinate $i \in \{1,\dots,n\}$. We write $I_a^\alpha$ for the multidimensional Riemann--Liouville integral with left origin at the point $a=(a_1,\ldots,a_n)\in\mathbb{R}^n$, for $\alpha\in [0,\infty)^n$; in particular, $I^\alpha=I_{(0,\ldots,0)}^\alpha$ \cite[Section~2]{meu}.

\begin{theorem} \label{CartMc_several}
Let $E$ be the space $\mathrm{L}^p[0,1]^n$ ($1\leq p<\infty$), or $C[0,1]^n$. Then there is precisely one family $\{J^\alpha:\alpha \in [0,\infty)^n\}$ of real-valued operators on $E$ satisfying the following conditions:
\begin{itemize}
\item Interpolation of the usual integral: For all $i\in\{1,\ldots,n\}$, we have $J^{e_i}=I^{e_i}$ on $E$, where $e_i$ is the $i$th canonical vector of $\mathbb{R}^n$;
\item Index law: $J^{\alpha+\beta}=J^\alpha\circ J^\beta$ on $E$ for all $\alpha,\beta\in [0,\infty)^n$;
\item Continuity: The map $\alpha\mapsto J^{\alpha}$ is continuous from $(0,\infty)^n$ to $\mathcal{B}(E)$ for some Hausdorff topology on $\mathcal{B}(E)$ weaker than the norm topology.
\end{itemize}
We find that $J^\alpha=I^\alpha$ is the multidimensional Riemann--Liouville integral on $E$, for all $\alpha\in (0,\infty)^n$.
\end{theorem}

\begin{proof}
Due to the continuity property in a topology that is Hausdorff (which ensures the uniqueness of the limit) and the density of $(\mathbb{Q}^+)^n$ in $(\mathbb{R}^+)^n$, it is clear that it is enough to show the result for $\alpha \in (\mathbb{Q}^+)^n$. Equivalently, it suffices to show this for any vector of the form $\alpha=m/p$, where $m \in (\mathbb{Z}^+)^n$ and $p \in \mathbb{Z}^+$ (we can assume a common denominator $p$ by rescaling the numerator and denominator if necessary). 

First, we note that we can show that $J^{\alpha}$ commutes with any $n$-dimensional convolution operator that has an integrable kernel. Indeed, by the index-law and interpolation properties, we know that $J^m=I^m$ for any $m \in (\mathbb{Z}^+)^n$ (observe that using $[0,\infty)^n$ for the orders in the index law is key and $(0,\infty)^n$ is not sufficient). Since $J^{\alpha}\circ I^m=J^{\alpha}\circ J^m=J^{m+\alpha}=J^m\circ J^{\alpha}=I^m\circ J^{\alpha}$ and $J^{\alpha}$ is linear, we have that $J^{\alpha}$ commutes with any convolution operator with a (multivariate) polynomial kernel: $J^{\alpha}[\int_0^t p(s)x(t-s)\mathrm{d}s]=\int_0^t p(s)J^\alpha x(t-s)\mathrm{d}s$, where $p$ is any polynomial, $t=(t_1,\ldots,t_n)$, and the integral is over $[0,t]:=[0,t_1]\times\cdots\times [0,t_n]\subseteq [0,1]^n$. By a density argument, $J^\alpha$ commutes with any convolution operator $M_h$ that has a kernel $h\in E$. In other words, $J^\alpha[\int_0^t h(s)x(t-s)\mathrm{d}s]=\int_0^t h(s)J^\alpha x(t-s)\mathrm{d}s$ for $[0,t]=[0,t_1]\times\cdots\times [0,t_n]\subseteq [0,1]^n$.

Now, we rewrite $\alpha=m/p$ and mimic the technique in \cite{cart}, so the main idea is to use the factorization $(J^{m/p}-\eta_1 I^{m/p})\circ \cdots \circ (J^{m/p}-\eta_p I^{m/p})=J^m-I^m=0$, where each $\eta_r$ is a complex $p$-root of unity. 

Initially, we will assume that $J^{m/p}$ is defined via a kernel $K_{m/p}$ for all $m \in (\mathbb{Z}^+)^n$ and $p \in \mathbb{Z}^+$, that is, $J^{m/p}(f)=K_{m/p}*f$. Therefore, we can rewrite the factorization as $M_h=0$, where $h=(K_{m/p}-\eta_1 g_{m/p})*\cdots\ast(K_{m/p}-\eta_p g_{m/p})$, the symbol $\ast$ indicates a convolution, and $g_\alpha$ is the multivariate Riemann--Liouville kernel of order $\alpha$. Of course, two factors $i$ and $j$ cannot vanish simultaneously on $[0,\lambda]^n$ for any $\lambda>0$ since, in that case, we would have $K_{m/p}-\eta_i g_{m/p}=K_{m/p}-\eta_j g_{m/p}$ and $g_{m/p}=0$ in such a box. Therefore, the Titchmarsh theorem in several variables \cite{Lions}, which states that $\mathrm{ch}(\mathrm{supp}(h_1\ast h_2))=\mathrm{ch}(\mathrm{supp}(h_1))+\mathrm{ch}(\mathrm{supp}(h_2))$, ensures that a factor has to be identically zero on $[0,1]^n$ because the other factors contain $(0,\ldots,0)$ in their support. Here, $\mathrm{ch}(U)$ refers to the convex hull of the set $U$, which means the smallest convex set that contains $U$. Since $K_{m/p}$ is real, we need to have $K_{m/p}\equiv \pm g_{m/p}$ and $J^{m/p}=\pm I^{m/p}$, but the choice of the ``$-$'' sign is impossible by the index law with the orders $m/(2p)$. Hence, $K_{m/p}= g_{m/p}$ and $J^{m/p}=I^{m/p}$ for all $m \in (\mathbb{Z}^+)^n$ and $p \in \mathbb{Z}^+$.

Let $m \in (\mathbb{Z}^+)^n$ and $p \in \mathbb{Z}^+$. If $J^{m/p}$ is not defined by a convolution, we consider $S:=I^{m} \circ J^{m/p}$. We will see that $S$ is defined by a convolution and, after that, the conclusion will follow. We have $Sf=(I^m \circ J^\alpha)(f)=(J^\alpha \circ I^m)(f)=J^{\alpha}(g_m \ast f)=(J^{\alpha} \circ M_f)(g_m)=(M_f \circ J^{\alpha})(g_m)=M_f(J^{\alpha}(g_m))$, but the last expression can be seen as the convolution operator with kernel $G:=J^{\alpha}(g_m)$ acting on $f$, so $S f= M_{G}f$. Now, the conclusion is easily obtained due to commutativity, as it implies $S^p=I^{p m}\circ J^{m}=I^{(p+1)m}=(I^{(p+1)m/p})^p$. Since $S$ is a convolution operator and satisfies the index-law property, we can apply the previous factorization argument to deduce $S=I^{(p+1)m/p}=I^{m} \circ I^{m/p}$ and $I^{m} \circ J^{m/p}=I^{m} \circ I^{m/p}$. As $I^m$ is injective (for example, differentiate $m$ times to the left), we conclude that $J^{m/p}=I^{m/p}$.
\end{proof}

\begin{remark} \normalfont 
The proof is identical if we replace the hyper-rectangle $[0,1]^n$ by an arbitrary product of compact intervals.
\end{remark}

\begin{remark}[Complex-valued functions] \normalfont
All results of this paper apply when the image of $f$ takes complex values, if $K_\alpha$ is always real-valued or $J^\alpha g$ is real-valued when $g$ is real-valued. For example, in the proof of Theorem~\ref{CartMc_several}, we equally obtain $K_{m/p}\equiv \pm g_{m/p}$. In the context of Theorem~\ref{th1}, we will have $\mathrm{Re}(J^\alpha f)=J^\alpha (\mathrm{Re} f)$ and $\mathrm{Im}(J^\alpha f)=J^\alpha (\mathrm{Im} f)$, because $K_\alpha$ is real-valued; therefore, our result holds for $\mathrm{Re}(J^\alpha)$ and $\mathrm{Im}(J^\alpha)$, and consequently for $J^\alpha$. On the other hand, in the setting of Theorem~\ref{th1_cor33}, we will have $\Lambda^\alpha f=(J^\alpha 1)\ast f$, where $K_\alpha=J^\alpha 1$ is real-valued because $g=1$ is real-valued.
\end{remark}

\section{Characterization of the Riemann--Liouville integral in dimension $1$} \label{seci4}

In this section, we propose an alternative characterization for the Riemann--Liouville integral, different from the one presented in \cite{cart} and significantly connected with Section~\ref{sec_neces}. This new approach is based on the Laplace transform \cite[Appendix D.3]{diethelm} and the Cauchy functional equation \cite[Chapter~3]{saaty}. At the end of the section, we recover the characterization in \cite{cart} as a corollary.

Our functions will be real-valued of real variable. First, we use the domain $(0,\infty)$ for the kernel because it is the natural region for the Laplace transform, and later we adapt the proof to a bounded interval $(0,T)$, giving the original Cartwright--McMullen theorem~\cite{cart} as a consequence. We also start with a convolution operator, because it is amenable to computing with Laplace transforms, and later we adapt the proof for the non-convolution case. We do not use the Titchmarsh theorem at any step, compared to the proof by Cartwright and McMullen.

We let the hat notation indicate the Laplace transform on $(0,\infty)$. We denote by $\mathcal{L}$ the set of Laplace transformable functions $f:(0,\infty)\rightarrow\mathbb{R}$, with $\widehat{f}(x):=\int_0^\infty f(s)\mathrm{e}^{-sx}\mathrm{d}s$ well-defined in the Lebesgue sense for all $x>0$. Note that $f\in\mathrm{L}^1(0,T)$ for all $T>0$.

For example, if $f\in\mathrm{L}^1(0,2)$ and $f=0$ on $(2,\infty)$, then $f\in\mathcal{L}$. In fact, if $f$ is exponentially bounded on $(2,\infty)$ and integrable on every bounded set, then $f\in\mathcal{L}$. The kernel of the one-dimensional Riemann--Liouville integral $I^\alpha$ with left origin at $0$, $x^{\alpha-1}$, is in $\mathcal{L}$ for $\alpha\in (0,\infty)$.

\subsection{The case of convolution operators}

\begin{theorem}[Dimension $1$ for convolution operators] \label{th1}
Let $J^\alpha f=K_\alpha\ast f$ be a family of operators defined by convolution, where $K_\alpha,f:(0,\infty)\rightarrow\mathbb{R}$, $\alpha\in (0,\infty)$, and $K_\alpha,f\in\mathcal{L}$. Suppose: 
\begin{itemize}
\item Identity: $K_1=1$ on $(0,\infty)$, i.e., $J^1$ is the ordinary integral;
\item Index law: $J^{\alpha+\beta}=J^\alpha\circ J^\beta$ on $\mathcal{L}$ for all $\alpha,\beta\in (0,\infty)$; 
\item Continuity: For all $x\in (0,\infty)$, the map $\alpha\mapsto \widehat{K}_\alpha(x)$ from $(0,\infty)$ to $\mathbb{R}$ is continuous at one point. (This condition is equivalent to: For all $f\in\mathcal{L}$ and $x\in (0,\infty)$, the map $\alpha\mapsto \widehat{J^{\alpha}f}(x)$ from $(0,\infty)$ to $\mathbb{R}$ is continuous at one point.)
\end{itemize}
Then $J^\alpha=I^\alpha$ is the Riemann--Liouville integral on $\mathcal{L}$, for all $\alpha\in (0,\infty)$.
\end{theorem}
\begin{proof}
Let $R_\alpha=\widehat{K}_\alpha$. For all $f\in\mathcal{L}$ and all $x>0$, we have $R_{\alpha+\beta}(x)\widehat{f}(x)=(J^{\alpha+\beta}f)^\wedge(x)=(J^\alpha\circ J^\beta f)^\wedge(x)=R_\alpha(x)R_\beta(x)\widehat{f}(x)$. Since $f$ is arbitrary, we obtain $R_{\alpha+\beta}(x)=R_\alpha(x)R_\beta(x)$. From $R_1(x)=1/x\neq0$ and this index law, it is clear that $R_\gamma(x)$ cannot be $0$ for any $\gamma>0$. In fact, $R_\gamma(x)=(R_{\gamma/2}(x))^2>0$. Then, applying logarithms, $\log R_{\alpha+\beta}(x)=\log R_\alpha(x)+\log R_\beta(x)$. In terms of the fractional order, we have a Cauchy functional equation with a continuous function at one point. Corollary \ref{linear} implies that $\log R_\alpha(x)=d(x)\alpha$, for some function $d$ defined in $(0,\infty)$ and independent of $\alpha$. We obtain $R_\alpha(x)=\exp(d(x)\alpha)$. From $R_1(x)=1/x$, we have $\exp(d(x))=1/x$, so $d(x)=\log (1/x)$ and $R_\alpha(x)=1/x^\alpha$. In conclusion, $\widehat{J^\alpha f}(x)=\widehat{f}(x)/x^\alpha=\widehat{I^\alpha f}(x)$ on $(0,\infty)$, which implies that $J^\alpha f(x)=I^\alpha f(x)$ almost everywhere on $(0,\infty)$, as desired.
\end{proof}

\subsection{The case of non-convolution operators} \label{subs_no_lin1}

\begin{theorem}[Dimension $1$, general case] \label{th1_cor33}
Let $J^\alpha$ be a family of linear operators defined for functions $f:(0,\infty)\rightarrow\mathbb{R}$ in $\mathcal{L}$ and $\alpha\in (0,\infty)$, such that $J^\alpha:\mathcal{L}\rightarrow \mathcal{L}$ is continuous for $\|\cdot\|_{\mathrm{L}^1(0,T)}$ and any $T>0$ (i.e., for each $T>0$, we have $\|J^\alpha f_1-J^\alpha f_2\|_{\mathrm{L}^1(0,T)}\rightarrow 0$ when $f_1,f_2\in\mathcal{L}$ and $\|f_1-f_2\|_{\mathrm{L}^1(0,T)}\rightarrow0$). Suppose: 
\begin{itemize}
\item Identity: $J^1$ is the ordinary integral;
\item Index law: $J^{\alpha+\beta}=J^\alpha\circ J^\beta$ on $\mathcal{L}$ for all $\alpha,\beta\in (0,\infty)$; 
\item Continuity: For all $x\in (0,\infty)$, the map $\alpha\mapsto \widehat{J^{\alpha}1}(x)$ from $(0,\infty)$ to $\mathbb{R}$ is continuous at one point.
\end{itemize}
Then $J^\alpha=I^\alpha$ is the Riemann--Liouville integral on $\mathcal{L}$, for all $\alpha\in (0,\infty)$.
\end{theorem}
\begin{proof}
The same density argument used in the second paragraph of the proof of Theorem \ref{CartMc_several} ensures that $J^\alpha$ commutes with any convolution operator that has a kernel in $\mathcal{L}$, for all $\alpha\in (0,\infty)^n$. In this case, the argument would be applied to an arbitrary finite interval $(0,T)$ and trivially extended to $\mathbb{R}^+$.

Observe that, if $f\in\mathcal{L}$, then $I^1\circ J^\alpha f=J^\alpha\circ I^1 f=J^\alpha[f\ast 1]=f\ast (J^\alpha 1)=(J^\alpha 1)\ast f=K_\alpha\ast f$, where $K_\alpha=J^\alpha 1\in\mathcal{L}$ satisfies the third condition of Theorem~\ref{th1}. The step $J^\alpha[f\ast 1]=f\ast (J^\alpha 1)$ uses the commutativity previously discussed: $J^\alpha[f\ast 1]=J^\alpha[M_f(1)]=(J^\alpha\circ M_f)(1)=(M_f\circ J^\alpha)(1)=M_f[J^\alpha 1]=f\ast (J^\alpha 1)$.

As in the proof of Theorem~\ref{th1}, we have that, for all $x>0$, $\widehat{K}_{\alpha+\beta+1}(x)\widehat{f}(x)=(I^1\circ J^{\alpha+\beta+1} f)^\wedge(x)=(J^{\alpha+\beta+2}f)^\wedge(x)=(J^{\alpha+1}\circ J^{\beta+1}f)^\wedge(x)=\widehat{K}_\alpha(x)\widehat{K}_\beta(x)\widehat{f}(x)$ for all $\alpha,\beta>0$, which gives the functional equation $\widehat{K}_{\alpha+\beta+1}(x)=\widehat{K}_\alpha(x)\widehat{K}_\beta(x)$. Let $R_\alpha=\widehat{K}_\alpha$ and $Q_A=R_{A-1}$, for $\alpha>0$ and $A>1$. Then, for $A=\alpha+1>1$ and $B=\beta+1>1$, we have $Q_{A+B}(x)=R_{A+B-1}(x)=R_{\alpha+1+\beta+1-1}(x)=R_\alpha(x) R_\beta(x)=Q_A(x) Q_B(x)$. It is easy to see that $Q_A(x)>0$. If we consider the function $\log Q_{\cdot}(x)$ for a fixed $x$, we can apply Corollary \ref{linear}, so $Q_A(x)=\exp(d(x)A)$. In other words, we get $R_\alpha(x)=\exp(d(x)(\alpha+1))$, for $\alpha>0$. We have $R_1(x)=\widehat{J^1 1}(x)=\widehat{I^1 1}(x)=1/x^2$, that is, $d(x)\cdot 2=\log (1/x^2)$ and $d(x)=\log (1/x)$, obtaining $R_\alpha(x)=1/x^{\alpha+1}$. As a consequence, $(I^1\circ J^\alpha f)^\wedge(x)=\widehat{f}(x)/x^{\alpha+1}=\widehat{I^{1+\alpha} f}(x)$ on $(0,\infty)$, which implies that $I^1\circ J^\alpha=I^{1+\alpha}$ on $\mathcal{L}$, that is, $I^1 \circ J^\alpha=I^1\circ I^\alpha$. Since $I^1$ is injective (differentiate almost everywhere to the left, for example), we deduce that $J^\alpha=I^\alpha$ on $\mathcal{L}$, as desired.
\end{proof}

\begin{remark} \normalfont
If we do not consider continuity assumptions, then we obtain the Cauchy functional equation $\log R_{\alpha+\beta}(x)=\log R_\alpha(x)+\log R_\beta(x)$, but with a non-continuous function $\alpha\mapsto R_\alpha(x)$. There exist such functions $R_\alpha$, for which the constructions in the literature rely on Hamel bases. Thus, the approach followed in Section~\ref{sec_neces} appears to be necessary.
\end{remark}

\begin{remark} \label{rm_QQ} \normalfont
If we remove the continuity assumption from all the previous statements, then we can ensure $J^\alpha=I^\alpha$ on $\mathcal{L}$ for all rational orders $\alpha\in (0,\infty)$. Indeed, the Cauchy functional equation $h(\alpha+\beta)=h(\alpha)+h(\beta)$ for $\alpha,\beta>0$ implies that $h(\alpha)=h(1)\alpha$ for all rational values $\alpha>0$, even if no conditions are given for $h$. If $h$ is continuous at one point, then the equality holds for every real number $\alpha>0$. With Hamel bases, it can be proved that there exists a discontinuous function $h:(0,\infty)\rightarrow\mathbb{R}$ such that $h(\alpha+\beta)=h(\alpha)+h(\beta)$ for $\alpha,\beta>0$ but it is not of the form $h(\alpha)=h(1)\alpha$ on $(0,\infty)$.
\end{remark}

\begin{remark} \normalfont
The continuity condition of Theorem~\ref{th1_cor33}, which is the most general result, is actually weaker than that of Theorem~\ref{th1}, due to the proof methods. In that sense, observe that Theorem~\ref{th1_cor33} uses the assumption that ``$\alpha\mapsto \widehat{J^{\alpha}1}(x)$ is continuous at one point'', whereas Theorem~\ref{th1} requires ``$\alpha\mapsto \widehat{J^{\alpha}f}(x)$ is continuous at one point for all $f\in\mathcal{L}$''.
\end{remark}

In order to derive the result for bounded intervals, we first need the following auxiliary lemma.

\begin{lemma} \label{auxiliar}
Fix $T\in (0,\infty)$. Let $J^\alpha$ be a family of continuous linear operators $\mathrm{L}^1(0,T)\rightarrow \mathrm{L}^1(0,T)$, for $\alpha\in (0,\infty)$. Suppose:
\begin{itemize}
\item Identity: $J^1$ is the ordinary integral;
\item Index law: $J^{\alpha+\beta}=J^\alpha\circ J^\beta$ on $\mathrm{L}^1(0,T)$ for all $\alpha,\beta\in (0,\infty)$; 
\item Continuity: The map $\alpha\mapsto J^{\alpha} 1$ from $(0,\infty)$ to $\mathrm{L}^1(0,T)$ is continuous at one point. (This condition is weaker than the global continuity $(0,\infty)\rightarrow\mathcal{B}(\mathrm{L}^1(0,T))$ required by Cartwright and McMullen.)
\end{itemize}
Then, for every $\alpha>0$, there exists $c_\alpha\in\mathbb R$
such that $J^\alpha=c_\alpha I^\alpha$.
\end{lemma}
\begin{proof}
The idea is to deduce the identity
\begin{equation} \label{ODE}
t(J^\alpha1)(t)=(\alpha+1)\int_0^t(J^\alpha1)(s)\mathrm{d}s \quad\text{for a.e. }t\in(0,T).
\end{equation}
After the deduction, the theorem follows because~\eqref{ODE} can be seen as a differential equation with respect to the absolutely continuous function $x(t)=(I^1\circ J^\alpha 1)(t)=\int_0^t (J^\alpha 1)(s)\mathrm{d}s$, where $tx'(t)=(\alpha+1)x(t)$, so the solution can be written as
\[(I^1\circ J^\alpha1)(t)=c_\alpha\frac{t^{\alpha+1}}{\Gamma(\alpha+2)}
=c_\alpha(I^{\alpha+1}1)(t)\] for some $c_\alpha\in\mathbb R$. Since $I^1$ is injective, $J^\alpha 1=c_\alpha I^\alpha 1$. Therefore, $I^1\circ J^\alpha f=J^\alpha\circ I^1 f=J^\alpha[1\ast f]=(J^\alpha1)\ast f=c_\alpha(I^\alpha1)\ast f=c_\alpha I^\alpha[1\ast f]=c_\alpha I^\alpha \circ I^1 f=c_\alpha I^1\circ I^\alpha f$, and canceling the injective operator $I^1$ would prove the result. We use the fact that $J^\alpha$ and $I^\alpha$ commute with convolution operators having an integrable kernel on $(0,T)$, as in the second paragraph of the proof of Theorem~\ref{CartMc_several}.

Therefore, we need to demonstrate~\eqref{ODE}. It is natural to introduce the operator $N$ given by $(Nf)(t)=tf(t)$ to rewrite the desired expression as $(N\circ J^{\alpha}) (1)= (\alpha+1)(I^1\circ J^{\alpha})(1)$. We recall that, due to the index law, $I^1\circ J^{\alpha} = J^{\alpha}\circ I^1$, and  $N1=I^1 1$, so the previous expression is equivalent to $[N,J^{\alpha}](1)=\alpha I^1\circ J^{\alpha}(1)$. The notation $[N,J^\alpha]=N\circ J^\alpha-J^\alpha\circ N$ is the commutator of operators, which evaluates the ``distance'' to commutativity. In other words, we can define $E_{\alpha}(1):=[N,J^{\alpha}](1)-\alpha I^1\circ J^{\alpha}(1)$ and our pretension is to show that $E_{\alpha}1=0$. Let $E_{\alpha}:=[N,J^{\alpha}]-\alpha I^1\circ J^{\alpha}$.

Initially, we show $E_{\alpha}=0$ for $\alpha=1$. If we have in mind the Cauchy formula for repeated integration, $[N,I^1](f)=t\int_{0}^t f(s)\mathrm{d}s - \int_{0}^t sf(s)\mathrm{d}s=\int_{0}^t (t-s)f(s)\mathrm{d}s=I^2(f)$, so $E_1 = 0$.

Now, we show a kind of index law for $E_{\alpha}$. We have, by definition, $E_{\alpha+\beta}=N\circ J^{\alpha+\beta}-J^{\alpha+\beta}\circ N-(\alpha+\beta)I^{1}\circ J^{\alpha+\beta}$, which can be rearranged as $(N\circ J^{\alpha} - J^{\alpha}\circ N-\alpha I^1\circ J^{\alpha})\circ J^{\beta} + J^{\alpha}\circ(N\circ J^{\beta} - J^{\beta}\circ N-\beta I^1\circ J^{\beta})$, or, in a compact notation, $E_{\alpha+\beta}=E_{\alpha}\circ J^{\beta} + J^{\alpha}\circ E_{\beta}$. Therefore, since $E_{1}=0$, we know that $E_{n}=0$, when $n \in \mathbb{Z}^+$. 

Observe that if we can show that $E_{\alpha}$ and $J^{\beta}$ commute, we are done. The reason is that for $p/q \in \mathbb{Q}^+$ we would have $0=E_{p}=qJ^{(q-1)p/q}\circ E_{p/q}$ and, due to injectivity of $J^{(q-1)p/q}$, we would get $E_{p/q}=0$. The continuity at one point $\alpha_0$ for $\alpha\mapsto J^\alpha 1$ then ensures that $E_{\alpha_0}1=0$ and, finally, considering any $\alpha>0$ and a sequence of rational numbers $r_k \to \alpha + \alpha_0$ we would have $0=E_{r_k}1=E_{\alpha}( J^{r_k-\alpha}1)+J^\alpha ( E_{r_k-\alpha}1) \to E_{\alpha}\circ J^{\alpha_0}1=J^{\alpha_0}\circ E_{\alpha}1$, and, since $J^{\alpha_0}$ is injective, $E_{\alpha}1=0.$

Consequently, we only need to justify that $E_{\alpha}$ and $J^{\beta}$ commute. In order to see this fact, we recall that we already know that every bounded operator $A$ that commutes with $I^1$ commutes with a convolution by any kernel in $\mathrm{L}^1(0,T)$, so $I^1\circ A f=A(f*1)=f*(A1).$ Consequently, two operators $A$ and $B$ that commute with $I^1$ commute between themselves, since $(I^1)^2\circ A\circ Bf=(A1)*(B1)*f=(I^1)^2\circ B\circ Af$ and $(I^1)^2$ can be canceled at both sides, since it is injective. Thus, considering the fact that $J^{\beta}$ commutes with $I^1$, we need to show that $E_{\alpha}$ also commutes with $I^1$. In other words, the commutator $[E_{\alpha},I^1]= [[N,J^{\alpha}], I^1]- \alpha[I^1\circ J^{\alpha}, I^1]$ should be zero. Of course, the second addend is obviously zero, so it suffices to check the equality $[[N,J^{\alpha}], I^1]=0$. We now perform calculations: $[[N,J^{\alpha}], I^1]=[N\circ J^{\alpha} - J^{\alpha}\circ N, I^1]=[N\circ J^{\alpha}, I^1]- [J^{\alpha}\circ N, I^1].$ This expression can be rewritten due to a Leibniz-type rule for commutators and the fact that $[J^{\alpha},I^1]=0$: specifically,
$[N\circ J^{\alpha}, I^1]=N\circ[J^{\alpha},I^1]+ [N, I^1]\circ J^{\alpha}= [N, I^1]\circ J^{\alpha}$ and
$[J^{\alpha}\circ N, I^1]=J^{\alpha}\circ[N,I^1]+ [J^{\alpha}, I^1]\circ N=J^{\alpha}\circ[N,I^1]$.
Finally, after recalling that $[N,I^1]=I^2$, we see that both addends cancel, so we conclude.
\end{proof}

\begin{corollary}[Dimension~1, a Cartwright-McMullen theorem] \label{th_CM}
Fix $T\in (0,\infty)$. Let $J^\alpha$ be a family of continuous linear operators $\mathrm{L}^1(0,T)\rightarrow \mathrm{L}^1(0,T)$, for $\alpha\in (0,\infty)$. Suppose that:
\begin{itemize}
\item Identity: $J^1$ is the ordinary integral;
\item Index law: $J^{\alpha+\beta}=J^\alpha\circ J^\beta$ on $\mathrm{L}^1(0,T)$ for all $\alpha,\beta\in (0,\infty)$; 
\item Continuity: The map $\alpha\mapsto J^{\alpha} 1$ from $(0,\infty)$ to $\mathrm{L}^1(0,T)$ is continuous at one point. (This condition is weaker than the global continuity $(0,\infty)\rightarrow\mathcal{B}(\mathrm{L}^1(0,T))$ required by Cartwright and McMullen.)
\end{itemize}
Then $J^\alpha=I^\alpha$ is the Riemann-Liouville integral on $\mathrm{L}^1(0,T)$, for all $\alpha\in (0,\infty)$.
\end{corollary}
\begin{proof}
Let $\alpha_0>0$ be a continuity point of $\alpha\mapsto J^\alpha1$. By the preceding Lemma~\ref{auxiliar}, $J^\alpha=c_\alpha I^\alpha$ on $\mathrm{L}^1(0,T)$. The identity and index law imply $c_1=1$ and $c_{\alpha+\beta}=c_\alpha c_\beta$. Moreover, since $\int_0^T(I^{\alpha}1)(t)\mathrm{d}t
=T^{\alpha+1}/\Gamma(\alpha+2)$, we can express \begin{equation*}\label{c_alpha}
c_\alpha=\frac{\Gamma(\alpha+2)}{T^{\alpha+1}} \int_0^T(J^\alpha1)(t)\mathrm{d}t,
\end{equation*}
hence $\alpha\mapsto c_\alpha$ is continuous at $\alpha_0$.

For $f\in\mathcal L$, define $G^\alpha f=c_\alpha I^\alpha f$
on $(0,\infty)$. These operators map $\mathcal L$ into itself,
satisfy the identity and index law, and extend $J^\alpha$: $(G^\alpha f)|_{(0,T)} =J^\alpha(f|_{(0,T)}).$
For every $S>0$, they also satisfy
\[\|G^\alpha f\|_{L^1(0,S)} \leq \frac{|c_\alpha|S^\alpha}{\Gamma(\alpha+1)} \|f\|_{L^1(0,S)}.\]
Finally, for every $x>0$, \[\widehat{G^\alpha1}(x)=\frac{c_\alpha}{x^{\alpha+1}},\] which is continuous in $\alpha$ at $\alpha_0$.
Thus, Theorem \ref{th1_cor33} applies and gives $G^\alpha=I^\alpha$.
Consequently, $c_\alpha=1$ and $J^\alpha=I^\alpha$ on $\mathrm{L}^1(0,T)$.
\end{proof}

\begin{remark} \normalfont 
Instead of concluding the previous result from the one for unbounded intervals (which was based on the Laplace transform), we could have used a development involving Lemma~\ref{auxiliar}, the continuity at $\alpha_0$, and the identities $c_1=1$ and $c_{\alpha+\beta}=c_\alpha c_\beta$. The argument for the Cauchy equation would show that $c_{\alpha}$ is determined and equal to $1$.
\end{remark}

\begin{remark} \normalfont
Like in Remark~\ref{rm_QQ}, if we remove the continuity condition from Corollary~\ref{th_CM}, then we have $J^\alpha=I^\alpha$ on $\mathrm{L}^1(0,T)$, for all rational values $\alpha\in (0,\infty)$.
\end{remark}

\begin{corollary}[Dimension~$1$, a Cartwright--McMullen theorem] \label{th_CM22111}
Fix $a,T\in\mathbb{R}$, $a<T$. Let $J_a^\alpha$ be a family of continuous linear operators $\mathrm{L}^1(a,T)\rightarrow \mathrm{L}^1(a,T)$, for $\alpha\in (0,\infty)$. Suppose: 
\begin{itemize}
\item Identity: $J_a^1$ is the ordinary integral;
\item Index law: $J_a^{\alpha+\beta}=J_a^\alpha\circ J_a^\beta$ on $\mathrm{L}^1(a,T)$ for all $\alpha,\beta\in (0,\infty)$; 
\item Continuity: The map $\alpha\mapsto J_a^{\alpha} 1$ from $(0,\infty)$ to $\mathrm{L}^1(a,T)$ is continuous at one point. (This condition is weaker than the global continuity $(0,\infty)\rightarrow\mathcal{B}(\mathrm{L}^1(a,T))$ required by Cartwright and McMullen.)
\end{itemize}
Then $J_a^\alpha=I_a^\alpha$ is the Riemann--Liouville integral on $\mathrm{L}^1(a,T)$, for all $\alpha\in (0,\infty)$.
\end{corollary}
\begin{proof}
We just consider the shift operator $\mathcal{S}f(x)=f(x-a)$ from $\mathrm{L}^1(0,T-a)$ to $\mathrm{L}^1(a,T)$ and define the conjugation relation $J^{\alpha}=\mathcal{S}^{-1}\circ J_a^\alpha \circ \mathcal{S}$ on $\mathrm{L}^1(0,T-a)$. This operator satisfies the conditions in Corollary~\ref{th_CM}, which gives $J^\alpha=I^\alpha$ on $\mathrm{L}^1(0,T-a)$. Consequently, $J_a^\alpha=\mathcal{S}\circ J^\alpha \circ \mathcal{S}^{-1}=\mathcal{S}\circ I^\alpha \circ \mathcal{S}^{-1}=I_a^\alpha$ on $\mathrm{L}^1(a,T)$.
\end{proof}

\section{Characterization of the Riemann--Liouville integral in dimension $n>1$} \label{seci5}

We extend the ideas of the previous section to the $n$-dimensional setting, $n>1$. The main issue in this case is that the kernels of the ordinary integrals are not proper functions, but delta distributions. That is, in dimension one, the kernel of $\int_0^x f(y)\mathrm{d}y$ is $K\equiv 1$, which is a proper simple function; however, in dimension $n>1$, the kernel of $\int_0^{x_1} f(y_1,x_2,\ldots,x_n)\mathrm{d}y_1$ is $K(y_1,\ldots,y_n)=\delta_0(y_2)\cdots \delta_0(y_n)$, where $\delta_0$ is the Dirac delta function. 

Our functions will be real-valued of real variable. We let the hat notation indicate the Laplace transform on $(0,\infty)^n$. We denote by $\mathcal{L}$ the set of Laplace transformable functions $f:(0,\infty)^n\rightarrow\mathbb{R}$, with $\widehat{f}(x):=\int_{(0,\infty)^n} f(s)\mathrm{e}^{-s\cdot x}\mathrm{d}s$ well-defined in the Lebesgue sense for all $x\in (0,\infty)^n$. Note that $f\in\mathrm{L}^1(0,T)$ for all $T\in (0,\infty)^n$. Here we denote the box $(0,T)=(0,T_1)\times\cdots (0,T_n)$. We write $I^\alpha$ for the multidimensional Riemann--Liouville integral with left origin at $(0,\ldots,0)$, for $\alpha\in [0,\infty)^n$.

The results provide an argumentation that differs completely from Section~\ref{sec3d}.

\subsection{The case of convolution operators}

We give two versions of the results because the ``identity hypothesis'' can be viewed with respect to different types of $n$-tuples. 

\begin{theorem}[Dimension $n$ for convolution operators, version~1] \label{thn}
Let $J^\alpha f=K_\alpha\ast f$ be a family of operators defined by convolution, where $K_\alpha,f:(0,\infty)^n\rightarrow\mathbb{R}$, $\alpha\in (0,\infty)^n$, $f\in\mathcal{L}$, and $K_\alpha \in\mathcal{L}$ for all $\alpha\in (0,\infty)^n$. Suppose: 
\begin{itemize}
\item Identity: For all $i\in\{1,\ldots,n\}$, we have $\widehat{K}_{\alpha}(x)\rightarrow 1/x_i$ on $(0,\infty)^n$ when $\alpha_j\rightarrow 0$ for $j\neq i$ and $\alpha_i=1$; 
\item Index law: $J^{\alpha+\beta}=J^\alpha\circ J^\beta$  on $\mathcal{L}$ for all $\alpha,\beta\in (0,\infty)^n$; 
\item Continuity: For all $x\in (0,\infty)^n$, the map $\alpha\mapsto \widehat{K}_\alpha(x)$ from $(0,\infty)^n$ to $\mathbb{R}$ is continuous at one point. (This condition is equivalent to: For all $f\in\mathcal{L}$ and $x\in (0,\infty)^n$, the map $\alpha\mapsto \widehat{J^{\alpha}f}(x)$ from $(0,\infty)^n$ to $\mathbb{R}$ is continuous at one point.)
\end{itemize}
Then $J^\alpha=I^\alpha$ is the multidimensional Riemann--Liouville integral on $\mathcal{L}$, for all $\alpha\in (0,\infty)$. 
\end{theorem}
\begin{proof}
Let $R_\alpha=\widehat{K}_\alpha$ for $\alpha\in (0,\infty)^n$. For all $f\in\mathcal{L}$ and all $x\in (0,\infty)^n$, we have $R_{\alpha+\beta}(x)\widehat{f}(x)=(J^{\alpha+\beta}f)^\wedge(x)=(J^\alpha\circ J^\beta f)^\wedge(x)=R_\alpha(x)R_\beta(x)\widehat{f}(x)$. Since $f$ is arbitrary, we obtain $R_{\alpha+\beta}(x)=R_\alpha(x)R_\beta(x)$, for all $\alpha,\beta\in (0,\infty)^n$. Note that $R_\alpha(x)=(R_{\alpha/2}(x))^2>0$ (it cannot be $0$ by the first condition). Then $\log R_{\alpha+\beta}(x)=\log R_\alpha(x)+\log R_\beta(x)$. We have a Cauchy functional equation with respect to the orders in $(0,\infty)^n$, with continuity at one point. By Corollary \ref{linear}, we know that $\log R_\alpha(x)=d(x)\cdot\alpha$, where $d:(0,\infty)^n\rightarrow \mathbb{R}^n$ is independent of $\alpha$ and the ``dot'' is the Euclidean product. We obtain $R_\alpha(x)=\exp(d(x)\cdot \alpha)$, for $\alpha\in (0,\infty)^n$. Using the first assumption, we compute the limit when $\alpha_j\rightarrow 0$ for $j\neq i$ and $\alpha_i=1$: $\exp(d_i(x))=1/x_i$, so $d_i(x)=\log (1/x_i)$ and $R_\alpha(x)=1/(x_1^{\alpha_1}\cdots x_n^{\alpha_n})$. Thus, $\widehat{J^\alpha f}(x)=\widehat{f}(x)/(x_1^{\alpha_1}\cdots x_n^{\alpha_n})=\widehat{I^\alpha f}(x)$ on $(0,\infty)^n$, which implies that $J^\alpha f(x)=I^\alpha f(x)$ almost everywhere on $(0,\infty)^n$, as desired.
\end{proof}

\begin{theorem}[Dimension $n$ for convolution operators, version~2] \label{thn22}
Let $J^\alpha f=K_\alpha\ast f$ be a family of operators defined by convolution, where $K_\alpha,f:(0,\infty)^n\rightarrow\mathbb{R}$, $\alpha\in (0,\infty)^n$, $f\in\mathcal{L}$, and $K_\alpha \in\mathcal{L}$ for all $\alpha\in (0,\infty)^n$. Suppose: 
\begin{itemize}
\item Identity: $K_\alpha(x)=x_1^{\alpha_1-1}\cdots x_n^{\alpha_n-1}/(\Gamma(\alpha_1)\cdots\Gamma(\alpha_n))$ for all $x\in(0,\infty)^n$, for $n$ tuples $\alpha^{(j)}\in (0,\infty)$, $j=1,\ldots,n$, that are linearly independent. That is, the kernel $K_\alpha$ coincides with the Riemann--Liouville kernel on $(0,\infty)^n$ (i.e., $J^\alpha=I^\alpha$) at $n$ fractional-order vectors that are linearly independent. 
\item Index law: $J^{\alpha+\beta}=J^\alpha\circ J^\beta$ on $\mathcal{L}$ for all $\alpha,\beta\in (0,\infty)^n$; 
\item Continuity: For all $x\in (0,\infty)^n$, the map $\alpha\mapsto \widehat{K}_\alpha(x)$ from $(0,\infty)^n$ to $\mathbb{R}$ is continuous at one point. (This condition is equivalent to: For all $f\in\mathcal{L}$ and $x\in (0,\infty)^n$, the map $\alpha\mapsto \widehat{J^{\alpha}f}(x)$ from $(0,\infty)^n$ to $\mathbb{R}$ is continuous at one point.)
\end{itemize}
Then $J^\alpha=I^\alpha$ is the multidimensional Riemann--Liouville integral on $\mathcal{L}$, for all $\alpha\in (0,\infty)$. 
\end{theorem}
\begin{proof}
It is analogous to that of Theorem~\ref{thn}. We arrive at $R_\alpha(x)=\exp(d(x)\cdot \alpha)$, for $\alpha\in (0,\infty)^n$. Changing $\alpha$ by $\alpha^{(j)}$, we obtain $\exp(d(x)\cdot \alpha^{(j)})=1/(x_1^{\alpha_1^{(j)}}\cdots x_n^{\alpha_n^{(j)}})$, that is, $d(x)\cdot \alpha^{(j)}=\log(1/(x_1^{\alpha_1^{(j)}}\cdots x_n^{\alpha_n^{(j)}}))$. This is a linear system with a unique solution because of the linear independence. Observe that $d(x)=(\log (1/x_1),\ldots,\log (1/x_n))$ is a solution and, hence, the unique solution. Thus, $R_\alpha(x)=1/(x_1^{\alpha_1}\cdots x_n^{\alpha_n})$ for all $\alpha\in (0,\infty)$, as desired.
\end{proof}

\begin{remark} \normalfont
The identity assumption of Theorem~\ref{thn} is essentially that of Theorem~\ref{thn22} when the vectors $\alpha^{(j)}$ form the canonical basis. However, in that case, $K_\alpha$ would be defined through Dirac delta functions; therefore, the assumptions of the theorems are stated differently.
\end{remark}

\subsection{The case of non-convolution operators} \label{subs_no_lin2}

\begin{theorem}[Dimension $n$, general case] \label{thn_g}
Let $J^\alpha$ be a family of operators defined for functions $f:(0,\infty)^n\rightarrow\mathbb{R}$ in $\mathcal{L}$ and $\alpha\in [0,\infty)^n$, such that $J^\alpha:\mathcal{L}\rightarrow \mathcal{L}$ is linear and continuous for $\|\cdot\|_{\mathrm{L}^1(0,T)}$ and any $T\in (0,\infty)^n$. Suppose: 
\begin{itemize}
\item Identity: For all $i\in\{1,\ldots,n\}$, we have $J^{e_i}=I^{e_i}$ on $\mathcal{L}$, where $e_i$ is the $i$th canonical vector of $\mathbb{R}^n$;
\item Index law: $J^{\alpha+\beta}=J^\alpha\circ J^\beta$  on $\mathcal{L}$ for all $\alpha,\beta\in [0,\infty)^n$; 
\item Continuity: For all $x\in (0,\infty)^n$, the map $\alpha\mapsto \widehat{J^{\alpha}1}(x)$ from $(0,\infty)^n$ to $\mathbb{R}$ is continuous at one point.
\end{itemize}
Then $J^\alpha=I^\alpha$ is the multidimensional Riemann--Liouville integral on $\mathcal{L}$, for all $\alpha\in [0,\infty)^n$. 
\end{theorem}
\begin{proof}
The same density argument used in the second paragraph of the proof of Theorem \ref{CartMc_several} ensures that $J^\alpha$ commutes with any $n$-dimensional convolution operator that has a kernel in $\mathcal{L}$, for all $\alpha\in [0,\infty)^n$. In this case, the argument would be applied to an arbitrary $(0,T)$.

Let $\Lambda^\alpha =I^{\bar{1}}\circ J^\alpha$ (the superscript of $I$ is formed by ones). Observe that, for $\alpha\in [0,\infty)^n$, $\Lambda^\alpha f=J^\alpha\circ I^{\bar{1}} f=J^\alpha[f\ast 1]=f\ast (J^\alpha 1)=(J^\alpha 1)\ast f=K_\alpha\ast f$, where $K_\alpha=J^\alpha 1\in\mathcal{L}$. The step $J^\alpha[f\ast 1]=f\ast (J^\alpha 1)$ uses the commutativity previously discussed.

As in the proof of Theorem~\ref{thn}, we have that, for all $x\in (0,\infty)^n$, $\widehat{K}_{\alpha+\beta+\bar{1}}(x)=\widehat{K}_\alpha(x)\widehat{K}_\beta(x)$ for all $\alpha,\beta\in [0,\infty)^n$, which is a functional equation. Let $R_\alpha=\widehat{K}_\alpha$ and $Q_A=R_{A-\bar{1}}$, for $\alpha\in [0,\infty)^n$ and $A\in [1,\infty)^n$. Then $Q_{A+B}(x)=R_{A+B-\bar{1}}(x)=R_{\alpha+\bar{1}+\beta+\bar{1}-\bar{1}}(x)=R_\alpha(x) R_\beta(x)=Q_A(x) Q_B(x)$, where $A=\alpha+\bar{1}\in [1,\infty)^n$ and $B=\beta+\bar{1}\in [1,\infty)^n$. By continuity at one point in $[1,\infty)^n$ and Corollary \ref{linear}, this fact implies that $Q_A(x)=\exp(d(x)\cdot A)$, that is, $R_\alpha(x)=\exp(d(x)\cdot(\alpha+\bar{1}))$, for $\alpha\in [0,\infty)^n$, where $d:(0,\infty)^n\rightarrow\mathbb{R}^n$ is independent of $\alpha$. We have $R_{e_i}(x)=\widehat{J^{e_i} 1}(x)=\widehat{I^{e_i} 1}(x)=1/x^{e_i+\bar{1}}$, where the multi-index power is understood as a product of the components as usual. That is, $d(x)\cdot (e_i+\bar{1})=\log (1/x^{e_i+\bar{1}})$. Note that $d(x)=(\log (1/x_1),\ldots,\log (1/x_n))$ is a solution and since the linear system has a unique solution, it is the solution. We obtain $R_\alpha(x)=1/x^{\alpha+\bar{1}}$. As a consequence, $\widehat{\Lambda^\alpha f}(x)=\widehat{f}(x)/x^{\alpha+\bar{1}}=\widehat{I^{\bar{1}+\alpha} f}(x)$ on $(0,\infty)^n$, which implies that $\Lambda^\alpha=I^{\bar{1}+\alpha}$ on $\mathcal{L}$, that is, $I^{\bar{1}}\circ J^\alpha=I^{\bar{1}} \circ I^\alpha$. Since $I^{\bar{1}}$ is injective (differentiate almost everywhere to the left, for example), we deduce that $J^\alpha=I^\alpha$ on $\mathcal{L}$, as desired.
\end{proof}

\begin{corollary}[Dimension~$n$, extension of the original Cartwright--McMullen theorem] \label{th_CM22}
Fix values $T_1,\ldots,T_n\in (0,\infty)$ and $(0,T)=(0,T_1)\times\cdots\times (0,T_n)$. Let $J^\alpha$ be a family of operators on $\mathrm{L}^1(0,T)$ for $\alpha\in [0,\infty)^n$, such that $J^\alpha:\mathrm{L}^1(0,T)\rightarrow \mathrm{L}^1(0,T)$ is linear and continuous. Suppose: 
\begin{itemize}
\item Identity: For all $i\in\{1,\ldots,n\}$, we have $J^{e_i}=I^{e_i}$ on $\mathrm{L}^1(0,T)$, where $e_i$ is the $i$th canonical vector of $\mathbb{R}^n$;
\item Index law: $J^{\alpha+\beta}=J^\alpha\circ J^\beta$ on $\mathrm{L}^1(0,T)$ for all $\alpha,\beta\in [0,\infty)^n$; 
\item Continuity: The map $\alpha\mapsto J^{\alpha} 1$ from $(0,\infty)^n$ to $\mathrm{L}^1(0,T)$ is continuous at one point. 
\end{itemize}
Then $J^\alpha=I^\alpha$ is the multidimensional Riemann--Liouville integral on $\mathrm{L}^1(0,T)$, for all $\alpha\in [0,\infty)^n$.
\end{corollary}
\begin{proof}
The argument is analogous to that of Lemma \ref{auxiliar} and Corollary~\ref{th_CM}, but applied to each $J^{\lambda_i e_i}$, where $\lambda_i \in \mathbb{R}^+$. The conclusion for $J^{\alpha}$, where $\alpha\in [0,\infty)^n$, follows from the index law.
\end{proof}

\begin{corollary}[Dimension~$n$, extension of the original Cartwright--McMullen theorem] \label{th_CM2244}
Fix values $a_1,\ldots,a_n,T_1,\ldots,T_n\in \mathbb{R}$, $a_j<T_j$, and $(a,T)=(a_1,T_1)\times\cdots\times (a_n,T_n)$. Let $J_a^\alpha$ be a family of operators on $\mathrm{L}^1(a,T)$ for $\alpha\in [0,\infty)^n$, such that $J_a^\alpha:\mathrm{L}^1(a,T)\rightarrow \mathrm{L}^1(a,T)$ is linear and continuous. Suppose: 
\begin{itemize}
\item Identity: For all $i\in\{1,\ldots,n\}$, we have $J_a^{e_i}=I_a^{e_i}$ on $\mathrm{L}^1(a,T)$, where $e_i$ is the $i$th canonical vector of $\mathbb{R}^n$;
\item Index law: $J_a^{\alpha+\beta}=J_a^\alpha\circ J_a^\beta$ on $\mathrm{L}^1(a,T)$ for all $\alpha,\beta\in [0,\infty)^n$; 
\item Continuity: The map $\alpha\mapsto J_a^{\alpha} 1$ from $(0,\infty)^n$ to $\mathrm{L}^1(a,T)$ is continuous at one point.
\end{itemize}
Then $J_a^\alpha=I_a^\alpha$ is the multidimensional Riemann--Liouville integral on $\mathrm{L}^1(a,T)$, for all $\alpha\in [0,\infty)^n$.
\end{corollary}
\begin{proof}
Proceed as in Corollary~\ref{th_CM22111}.
\end{proof}

\section{Characterization of the Riesz potential in dimension $n\geq 1$}

Now we consider the Riesz potential $\mathcal{I}^\alpha$ on $\mathbb{R}^n$, where $n\geq 1$ and $0<\alpha<n$. We use the convolution and the Fourier transform, indicated with a hat notation, on $\mathbb{R}^n$ \cite{samko_fcaa}. The set of rapidly decreasing functions or Schwartz functions is the Schwartz space, defined in \cite[Section~7.3]{rudin22}.

\begin{theorem} \label{th1RR}
Let $J^\alpha f=K_\alpha\ast f$ be a family of operators defined by convolution (in the sense of functions or distributions), where $K_\alpha,f:\mathbb{R}^n\rightarrow\mathbb{R}^n$, $0<\alpha<n$, $f$ is a Schwartz function, $K_\alpha$ is locally integrable, and $K_\alpha$ has a Fourier transform (in the sense of functions or distributions) that is defined for all $x\neq 0$. Suppose: 
\begin{itemize}
\item Identity: For some $0<\alpha_0<n$, $\widehat{K}_{\alpha_0}(x)=1/|x|^{\alpha_0}$ for all $x\neq 0$ (i.e., $J^{\alpha_0}=\mathcal{I}^{\alpha_0}$);
\item Index law: $J^{\alpha+\beta}=J^\alpha\circ J^\beta$ on the Schwartz space, for all $\alpha,\beta\in (0,n)$ such that $\alpha+\beta<n$; 
\item Continuity: For all $x\neq 0$, the map $\alpha\mapsto \widehat{K}_\alpha(x)$ from $(0,n)$ to $\mathbb{R}$ is continuous at one point. (This condition is equivalent to: For every Schwartz function $f$ and $x\neq 0$, the map $\alpha\mapsto \widehat{J^{\alpha}f}(x)$ from $(0,n)$ to $\mathbb{R}$ is continuous at one point.)
\end{itemize}
Then $J^\alpha=\mathcal{I}^\alpha$ is the Riesz potential on the Schwartz space, for all $\alpha\in (0,\infty)$.
\end{theorem}
\begin{proof}
Consider $R_\alpha=\widehat{K}_\alpha$. For every Schwartz function $f$ and for all $x\neq 0$, we have $R_{\alpha+\beta}(x)\widehat{f}(x)=(J^{\alpha+\beta}f)^\wedge(x)=(J^\alpha\circ J^\beta f)^\wedge(x)=R_\alpha(x)R_\beta(x)\widehat{f}(x)$. Since $f$ is arbitrary, we obtain $R_{\alpha+\beta}(x)=R_\alpha(x)R_\beta(x)$, for all $\alpha,\beta\in (0,n)$ such that $\alpha+\beta<n$. From $R_{\alpha_0}(x)=1/|x|^{\alpha_0}>0$ and this index law, it is clear that $R_\gamma(x)$ is positive for any $\gamma\in (0,n)$. We can use logarithms, $\log R_{\alpha+\beta}(x)=\log R_\alpha(x)+\log R_\beta(x)$, which gives a Cauchy functional equation on a restricted domain $(0,n)$. By Lemma~\ref{cfr}, it implies that $\log R_\alpha(x)=d(x)\alpha$, for some function $d:\mathbb{R}^n\backslash\{0\}\rightarrow\mathbb{R}$ independent of $\alpha\in (0,n)$. That is, $R_\alpha(x)=\exp(d(x)\alpha)$. From $R_{\alpha_0}(x)=1/|x|^{\alpha_0}>0$, we have $d(x)=\log (1/|x|)$ and $R_\alpha(x)=1/|x|^\alpha$. In conclusion, $\widehat{J^\alpha f}(x)=\widehat{f}(x)/|x|^\alpha=\widehat{\mathcal{I}^\alpha f}(x)$ for $x\neq 0$, which gives $J^\alpha f(x)=\mathcal{I}^\alpha f(x)$ almost everywhere on $\mathbb{R}^n$, as desired.
\end{proof}

\begin{remark} \normalfont
The ``identity hypothesis'' of Theorem~\ref{th1RR} can be replaced by $\lim_{\alpha\rightarrow n^-}\widehat{K}_{\alpha}(x)=1/|x|^{n}$ for all $x\neq 0$. Indeed, this assumption also gives $d(x)=\log (1/|x|)$.
\end{remark}

\begin{remark} \normalfont
For the case of non-convolution operators in Theorem~\ref{th1RR}, we cannot adapt the proof of Theorem~\ref{th1_cor33}, because of the restricted orders in the index law. There is no commutation property with convolution operators.
\end{remark}

\section{Characterization of the Riemann--Liouville integral under transmutation}

Transmutation is the relationship between two operators $X$ and $Y$ through a bijective operator $Z$ such that $X=Z\circ Y\circ Z^{-1}$. In this way, the properties of $X$ and $Y$ can be related. This idea has been of application in fractional calculus. In ordinary calculus, if we have $F(x)=\int f(x)\mathrm{d}x$,
then a change of variable $x=g(y)$ proceeds as follows: $\tilde{F}(y)=\int f(g(y))g'(y)\mathrm{d}y$ and $F(x)=\tilde{F}(g^{-1}(y))$,
where $g$ is a continuously differentiable bijective function. If $If(x)=\int f(x)\mathrm{d}x$, $I_g f(x)=\int f(x)\mathrm{d}g(x)$ is the integral with respect to $g$, and $Q_g f=f\circ g$ is the composition operator, then $I_g\circ Q_g=Q_g\circ I$,
so we have the transmutation relation $I_g=Q_g\circ I\circ Q_g^{-1}$. This approach can be extended, for example, to the Riemann--Liouville integral. The motivation of this section is the characterization of the new Riemann--Liouville integral with respect to a function, but compared with \cite{cao_jie,meu}, the following work relies on the fact that $g$ might not be continuously differentiable, because the change-of-variable formula can be written more abstractly in the context of measure theory.

\subsection{Transmutation for integrals with a kernel using pushforward measures}

In $\mathbb{R}^n$, $n\geq 1$, we use a generalized notation for boxes. For example, if $a=(a_1,\ldots,a_n)\in\mathbb{R}^n$ and $T=(T_1,\ldots,T_n)\in\mathbb{R}^n$, then $[a,T]=[a_1,T_1]\times\cdots\times [a_n,T_n]$. Similarly, $\int_a^T=\int_{a_1}^{T_1}\cdots \int_{a_n}^{T_n}$. In this way, we can mimic the standard notation in $\mathbb{R}$.

Let $\Xi$ be the set of functions $\phi:[a,T]\rightarrow\mathbb{R}^n$ such that $\phi(t)=(\phi_1(t_1),\ldots,\phi_n(t_n))$, where each component $\phi_j$ is a strictly increasing function on $[a_j,T_j]$. In particular, this property implies that $\phi$ is Borel-measurable with a Borel-measurable inverse function. Recall that an increasing function has bounded variation and is almost everywhere differentiable, but is not necessarily continuous on $[a,T]$ (under continuity, it is not necessarily absolutely continuous). We will also assume that, if $\phi\in\Xi$, then each $\phi_j$ is continuous on $[a_j,T_j]$.

Let $\mu$ be a (non-negative) $\sigma$-finite Borel measure on $[\phi(a),\phi(T)]$. The pushforward measure with respect to $\phi$ \cite[Section~3.6]{boga} is $\mu\circ \phi$ on $[a,T]$, defined by $(\mu\circ \phi)(E)=\mu(\phi(E))$ for Borel-measurable sets $E\subseteq [a,T]$ (note that $\phi(E)$ is a Borel set because the function $\phi^{-1}$ is Borel-measurable). Since $\phi$ is injective, the pushforward measure is $\sigma$-finite. (Recall that $\sigma$-finite measures are needed to apply Fubini's theorem.)

We define the general integral with a kernel and with respect to $\phi$.

\begin{definition} \label{def1}
The integral with kernel $K\in\mathrm{L}^1([0,\phi(T)-\phi(a)],\mathrm{d}\mu)$ and left origin $a\in\mathbb{R}^n$ with respect to $\phi\in\Xi$ is
\begin{equation}
 I_{a,\phi} g(t)=\int_a^t K(\phi(t)-\phi(s))g(s)\,\mathrm{d}(\mu\circ \phi)(s), \label{eq1}
\end{equation}
for $g\circ\phi^{-1}\in\mathrm{L}^1([\phi(a),\phi(T)],\mathrm{d}\mu)$ and $t\in [a,T]$. In particular, if $\phi$ is the identity function, then $K\in\mathrm{L}^1([0,T-a],\mathrm{d}\mu)$ and
\begin{equation} 
I_{a}g(t)=\int_{a}^{t} K(t-s)g(s)\,\mathrm{d}\mu(s), 
\label{eq2}
\end{equation}
for $g\in\mathrm{L}^1([a,T],\mathrm{d}\mu)$.
\end{definition}

It is important to express~\eqref{eq1} in terms of transmutation of operators to relate $I_{a,\phi}$ to the case $\phi=\text{identity}$. The key operator is $Q_\phi$.

\begin{definition} \label{def3}
The composition operator with respect to a function $\phi$ is $Q_\phi f=f\circ \phi$.
\end{definition}

The transmutation relation requires tools from measure theory; specifically, the change-of-variable formula in integration.

\begin{lemma} \label{lemMe}
Let $X_1$ and $X_2$ be two sets, where $X_1$ has a measure $\nu$ and $X_2$ has another measure $\eta$. Let $\rho:X_1\rightarrow X_2$ be a measurable function. Then $\nu\circ \rho^{-1}$ is the pushforward measure on $X_2$ and 
\begin{equation} \int_{X_2} g\,\mathrm{d}(\nu\circ \rho^{-1})=\int_{X_1} g\circ \rho\, \mathrm{d}\nu, \label{fam1} \end{equation}
for $g\in \mathrm{L}^1(X_2,\mathrm{d}(\nu\circ\rho^{-1}))$ (equivalently $g\circ\rho\in\mathrm{L}^1(X_1,\mathrm{d}\nu)$). \cite[Theorem~3.6.1]{boga}

If $\rho$ is injective with a measurable inverse function $\rho^{-1}$, then
\begin{equation} \int_{\rho(X_1)} g\,\mathrm{d}\eta=\int_{X_1}g\circ \rho\,\mathrm{d}(\eta\circ \rho). \label{fam2} \end{equation}

In particular, if $n=1$, $X_1=[a,T]$, $\rho$ is strictly increasing so that $\rho(X_1)=[\rho(a),\rho(T)]$, and $\nu=\eta$ are Lebesgue measures, then we obtain a Stieltjes integration
\begin{equation} \int_{\rho(a)}^{\rho(T)} g(t)\mathrm{d}t=\int_a^T g(\rho(t))\rho'(t)\mathrm{d}t \label{fam3} \end{equation}
when $\rho$ is continuously differentiable or is just absolutely continuous \cite[Theorem~3.7.1]{boga}, \cite[Theorem~7.26]{rudin} (for~\eqref{fam3}, it is not sufficient to have a strictly increasing continuous function).
\end{lemma}

\begin{remark} \normalfont
Some comments on Lemma~\ref{lemMe} are the following. On the one hand, expression~\eqref{fam2} follows from~\eqref{fam1} by taking $\nu=\eta\circ \rho$. To ensure that $\nu$ is a measure, the inverse function $\rho^{-1}$ must be measurable. On the other hand, identity~\eqref{fam3} is the standard change-of-variable theorem in real analysis, and the classical theory of transmutation in fractional calculus is based on it. Thus, Lemma~\ref{lemMe} shows that a measure-theoretic approach allows the extension of the classical theory of transmutation to abstract measures and measurable functions. Observe that it is very intuitive that $\mathrm{d}(\eta\circ \rho)(t)=\rho'(t)\mathrm{d}t$ when $\phi$ is ``smooth'', because $\mathrm{d}(\eta\circ \rho)(t)$ indicates a small change in the measure induced by $\rho$, namely $\Delta\rho(t)$, which is $\rho'(t)\Delta t$ in the limit according to the mean-value theorem.
\end{remark}

\begin{lemma} \label{qp}
For $\phi\in\Xi$, the operator of Definition~\ref{def3} is linear and continuous from the space $\mathrm{L}^1([\phi(a),\phi(T)],\mathrm{d}\mu)$ to $\mathrm{L}^1([a,T],\mathrm{d}(\mu\circ\phi))$, with norm equal to $1$. It is bijective with an inverse operator $Q_{\phi^{-1}}$.
\end{lemma}
\begin{proof}
Linearity is clear by the definition of addition of functions and multiplication by a scalar value. Let $f \in \mathrm{L}^1([\phi(a),\phi(T)],\mathrm{d}\mu)$. Then $\|Q_\phi f\|_{\mathrm{L}^1([a,T],\mathrm{d}(\mu\circ\phi))}=  \int_a^t |f\circ \phi|\mathrm{d}(\mu\circ\phi) 
=  \int_{\phi(a)}^{\phi(t)} |f|\mathrm{d}\mu=\|f\|_{\mathrm{L}^1([\phi(a),\phi(T)],\mathrm{d}\mu)}$,
as a consequence of Lemma~\ref{lemMe} with~\eqref{fam2}. This fact shows continuity with norm equal to $1$. Finally, $Q_{\phi}\circ Q_{\phi^{-1}}f=(f\circ\phi^{-1})\circ\phi=f$
and
$Q_{\phi^{-1}}\circ Q_\phi f=(f\circ\phi)\circ \phi^{-1}=f$. We note that $\phi$ is bijective and both $\phi^{-1}\circ \phi$ and $\phi\circ \phi^{-1}$ give the identity function because each component is continuous and strictly increasing.
\end{proof}

\begin{theorem} \label{th_prin}
For $\phi\in\Xi$ and the Lebesgue measure $\mu$, Definition~\ref{def1} is well-posed and the transmutation relation $I_{a,\phi}=Q_\phi\circ I_{\phi(a)}\circ Q_{\phi^{-1}}$ holds. The linear operator $I_{a,\phi}:\mathrm{L}^1([a,T],\mathrm{d}(\mu\circ\phi))\rightarrow \mathrm{L}^1([a,T],\mathrm{d}(\mu\circ\phi))$ is continuous.
\end{theorem}
\begin{proof}
Expression~\eqref{eq2}, in the context of $\phi=\text{identity}$, is similar to the convolution of two functions in $\mathrm{L}^1$, so it exists \cite[Proposition~1.25]{tese}:
\begin{equation}
\begin{split}
\| I_{a}g\|_{\mathrm{L}^1([a,T],\mathrm{d}\mu)}= {} & \int_a^T |I_a g(t)|\mathrm{d}\mu(t) \\
\leq {} & \int_a^T \int_{a}^{t} |K(t-s)||g(s)|\,\mathrm{d}\mu(s)\mathrm{d}\mu(t) \\
= {} & \int_a^T |g(s)|\left(\int_s^T |K(t-s)|\mathrm{d}\mu(t)\right)\mathrm{d}\mu(s) \\
= {} & \int_a^T |g(s)|\left(\int_0^{T-s} |K(\tau)|\mathrm{d}\mu(\tau)\right)\mathrm{d}\mu(s) \\
\leq {} & \|g\|_{\mathrm{L}^1([a,T],\mathrm{d}\mu)}\|K\|_{\mathrm{L}^1([0,T-a],\mathrm{d}\mu)}.
\end{split}
\label{moci}
\end{equation}
We used Fubini's theorem, with the $\sigma$-finite property. The Lebesgue measure $\mu$ is key to ensure that $\int_s^T |K(t-s)|\mathrm{d}\mu(t)=\int_0^{T-s} |K(\tau)|\mathrm{d}\mu(\tau)$, because it is invariant by translations. In particular, $I_a$ is continuous with norm less than or equal to $\|K\|_{\mathrm{L}^1([0,T-a],\mathrm{d}\mu)}$.

For expression~\eqref{eq1} with $\phi\in\Xi$, if $K\in\mathrm{L}^1([0,\phi(T)-\phi(a)],\mathrm{d}\mu)$ and $g\circ\phi^{-1}\in\mathrm{L}^1([\phi(a),\phi(T)],\mathrm{d}\mu)$, then $I_{\phi(a)} \circ Q_{\phi^{-1}}g$ makes sense and belongs to $\mathrm{L}^1([\phi(a),\phi(T)],\mathrm{d}\mu)$, according to the previous paragraph and~\eqref{moci} (with $g\circ\phi^{-1}$ and the interval $[\phi(a),\phi(T)]$). In addition, $Q_\phi \circ I_{\phi(a)}\circ Q_{\phi^{-1}}g(t)=  \int_{\phi(a)}^{\phi(t)}R(\phi(t)-s)(g\circ\phi^{-1})(s)\,\mathrm{d}\mu(s) 
= I_{a,\phi}g(t)$. Observe that the increasing property of $\phi$ guaranties $\phi([a,t])=[\phi(a),\phi(t)]$, and we used Lemma~\ref{lemMe} with~\eqref{fam2}. This shows that Definition~\ref{def1} is well-posed and the transmutation relation holds.

Note that $g\circ\phi^{-1}\in\mathrm{L}^1([\phi(a),\phi(T)],\mathrm{d}\mu)$ if and only if $g\in Q_\phi(\mathrm{L}^1([\phi(a),\phi(T)],\mathrm{d}\mu))$, which coincides with $\mathrm{L}^1([a,T],\mathrm{d}(\mu\circ\phi))$,
according to Lemma~\ref{qp}. Thus, $I_{a,\phi}$ is well-defined on $\mathrm{L}^1([a,T],\mathrm{d}(\mu\circ\phi))$. By Lemma~\ref{qp}, the composition
$ Q_\phi \circ I_{\phi(a)} \circ Q_{\phi^{-1}}:  \mathrm{L}^1([a,T],\mathrm{d}(\mu\circ\phi))\rightarrow \mathrm{L}^1([\phi(a),\phi(T)],\mathrm{d}\mu)
 \rightarrow \mathrm{L}^1([\phi(a),\phi(T)],\mathrm{d}\mu) \rightarrow \mathrm{L}^1([a,T],\mathrm{d}(\mu\circ\phi)) 
$
is continuous, which gives the desired continuity. The norm of $I_{a,\phi}$ is less than or equal to $\|K\|_{\mathrm{L}^1([0,\phi(T)-\phi(a)],\mathrm{d}\mu)}$.
\end{proof}

\subsection{Application: Transmutation in fractional calculus using pushforward measures}

When $\mu$ is the Lebesgue measure, the theory applies to the Riemann--Liouville integral with kernel $\frac{1}{\Gamma(\alpha)}\tau^{\alpha-1}$, where $\alpha\in (0,\infty)^n$ is the fractional order. We use multi-index notation: $\Gamma(\alpha)=\Gamma(\alpha_1)\cdots \Gamma(\alpha_n)$ and $\tau^{\alpha-1}=\tau_1^{\alpha_1-1}\cdots \tau_n^{\alpha_n-1}$.

\begin{definition}
The Riemann--Liouville integral with respect to $\phi\in\Xi$ is
\begin{equation} I_{a,\phi}^\alpha f(t)=\frac{1}{\Gamma(\alpha)}\int_a^t (\phi(t)-\phi(s))^{\alpha-1}f(s)\,\mathrm{d}(\mu\circ \phi)(s). \label{hol} \end{equation}
\end{definition}

When $\phi$ is the identity function, we recover the standard Riemann--Liouville integral operator $I_{a}^\alpha$. When $a=(0,\ldots,0)$, we use the notation $I_\phi^\alpha$ and $I^\alpha$.

We recall that, in the literature, the theory of transmutation has been applied when $\phi$ is a strictly increasing continuously differentiable function in $\mathbb{R}$, giving the particular case $\mathrm{d}(\mu\circ \phi)(s)=\phi'(s)\mathrm{d}s$.

The previous theory gives the function space domain and the transmutation relation for~\eqref{hol}.

\begin{corollary} \label{cor_indi}
The Riemann--Liouville integral with respect to $\phi\in\Xi$ is a continuous linear operator $I_{a,\phi}^\alpha:\mathrm{L}^1([a,T],\mathrm{d}(\mu\circ\phi))\rightarrow \mathrm{L}^1([a,T],\mathrm{d}(\mu\circ\phi))$.
It is connected to the standard Riemann--Liouville integral by transmutation:
$I_{a,\phi}^\alpha=Q_\phi \circ I_{\phi(a)}^\alpha\circ Q_{\phi^{-1}}$.
\end{corollary}
\begin{proof}
It is a consequence of Theorem~\ref{th_prin}.
\end{proof}

\begin{example} \normalfont
Suppose the integral equation 
\begin{equation} x(t)=x_a+\lambda I_{a,\phi}^\alpha x(t), \label{enunc} \end{equation}
based on~\eqref{hol}, where $\alpha\in (0,1)$, $x_a,\lambda\in\mathbb{R}$, and $t\in [a,T]$. Here, $\phi:[a,T]\rightarrow\mathbb{R}$ is strictly increasing and continuous. Applying $Q_{\phi^{-1}}$ on both sides, we obtain
\begin{equation} y(\tau)=x_a+\lambda  I_{\phi(a)}^\alpha y(\tau) \label{set} \end{equation}
on $[\phi(a),\phi(T)]$, where
\begin{equation} y=x\circ\phi^{-1}. \label{comsy} \end{equation}
It is well-known that the solution of~\eqref{set} is $y(\tau)=x_a\,\mathcal{E}_\alpha(\lambda (\tau-\phi(a))^\alpha)$,
where $\mathcal{E}_\alpha$ is the one-parameter Mittag-Leffler function \cite[Chapter~4]{diethelm}. Consequently, with~\eqref{comsy}, we obtain the solution of~\eqref{enunc}, which is continuous:
$x(t)=x_a\,\mathcal{E}_\alpha(\lambda (\phi(t)-\phi(a))^\alpha)$.
\end{example}

\subsection{Characterizations}

We use the ideas in \cite{cao_jie,meu}. The key point is that we have the same transmutation relations through the composition of operators, so the typical proofs on fractional calculus with respect to functions work in the setting of strictly increasing continuous functions. Now the Lebesgue spaces are given with respect to pushforward measures instead.

We illustrate one of the possible results. The approach is always similar.

\begin{theorem} \label{CartMc_several222}
Fix values $a_1,\ldots,a_n,T_1,\ldots,T_n\in \mathbb{R}$, $a_j<T_j$, and $(a,T)=(a_1,T_1)\times\cdots\times (a_n,T_n)$, $n\geq 1$. Let $\phi\in\Xi$ and $J_{a,\phi}^\alpha$ be a family of operators on $\mathrm{L}^1((a,T),\mathrm{d}(\mu\circ\phi))$ for $\alpha\in [0,\infty)^n$, such that $J_{a,\phi}^\alpha:\mathrm{L}^1((a,T),\mathrm{d}(\mu\circ\phi))\rightarrow \mathrm{L}^1((a,T),\mathrm{d}(\mu\circ\phi))$ is linear and continuous. Suppose: 
\begin{itemize}
\item Identity: For all $i\in\{1,\ldots,n\}$, we have $J_{a,\phi}^{e_i}=I_{a,\phi}^{e_i}$ on $\mathrm{L}^1((a,T),\mathrm{d}(\mu\circ\phi))$, where $e_i$ is the $i$th canonical vector of $\mathbb{R}^n$;
\item Index law: $J_{a,\phi}^{\alpha+\beta}=J_{a,\phi}^\alpha\circ J_{a,\phi}^\beta$ on $\mathrm{L}^1((a,T),\mathrm{d}(\mu\circ\phi))$ for all $\alpha,\beta\in [0,\infty)^n$; 
\item Continuity: The map $\alpha\mapsto J_{a,\phi}^{\alpha} 1$ from $(0,\infty)^n$ to $\mathrm{L}^1((a,T),\mathrm{d}(\mu\circ\phi))$ is continuous at one point.
\end{itemize}
Then $J_{a,\phi}^\alpha=I_{a,\phi}^\alpha$ is the new Riemann--Liouville integral~\eqref{hol} on $\mathrm{L}^1((a,T),\mathrm{d}(\mu\circ\phi))$, for all $\alpha\in [0,\infty)^n$.
\end{theorem}
\begin{proof}
Define $J_{\phi(a)}^\alpha=Q_{\phi}^{-1}\circ J_{a,\phi}^\alpha \circ Q_\phi$, which is a continuous linear operator in $\mathcal{B}(\mathrm{L}^1((\phi(a),\phi(T)),\mathrm{d}\mu))$ by Lemma~\ref{qp}. This family of operators satisfies the following properties:
\begin{itemize}
\item Identity: For all $i\in\{1,\ldots,n\}$, we have $J_{\phi(a)}^{e_i}=I_{\phi(a)}^{e_i}$ on $\mathrm{L}^1((\phi(a),\phi(T)),\mathrm{d}\mu)$. Indeed, $I_{\phi(a)}^\alpha=Q_{\phi}^{-1}\circ I_{a,\phi}^\alpha \circ Q_\phi$, by Corollary~\ref{cor_indi}.
\item Index law: $J_{\phi(a)}^{\alpha+\beta}=J_{\phi(a)}^\alpha\circ J_{\phi(a)}^\beta$ on $\mathrm{L}^1((\phi(a),\phi(T)),\mathrm{d}\mu)$ for all $\alpha,\beta\in [0,\infty)^n$.
\item Continuity: The map $\alpha\mapsto J_{\phi(a)}^{\alpha} 1$ from $(0,\infty)^n$ to $\mathrm{L}^1((\phi(a),\phi(T)),\mathrm{d}\mu)$ is continuous at one point. Indeed, $J_{\phi(a)}^{\alpha} 1=Q_{\phi}^{-1}\circ J_{a,\phi}^\alpha 1$, which satisfies the continuity requirement \cite[Lemma~1.5]{cao_jie}.
\end{itemize}
By Corollary~\ref{th_CM}, $J_{\phi(a)}^\alpha=I_{\phi(a)}^\alpha$ and $J_{a,\phi}^\alpha=I_{a,\phi}^\alpha$, as desired.
\end{proof}

\begin{remark} \normalfont
Related to transmutation is the fractional calculus with respect to a weight~\cite{meu}. In that case, we have a measure $\tilde{\mu}$ such that $\mathrm{d}\tilde{\mu}(t)=w(t)\mathrm{d}t$, where $\mathrm{d}t$ is the Lebesgue measure. The weight $w$ is the Radon--Nikodym derivative $\mathrm{d}\tilde{\mu}/\mathrm{d}t$ \cite[Section~3.2]{boga}, thus becoming a key concept in measure theory. However, in contrast to our preceding development with respect to $\phi$, a new theory with respect to $w$ is not really necessary compared to previous literature: In the literature of fractional calculus, the integrand $x$ has already been considered in the weighted space $\mathrm{L}^p_w[a,T]$, which corresponds to $\mathrm{L}^p([a,T],\mathrm{d}\tilde{\mu})$, and this fact is distinct from our previous exposition where $\mathrm{d}(\mu\circ\phi)(t)$ was not necessarily $\phi'(t)\mathrm{d}\mu(t)$. To summarize, we see that the fractional calculus with respect to a function is based on the ``pushforward measure'' and the fractional calculus with respect to a weight is based on the ``Radon--Nikodym derivative'', with the former requiring a new exposition compared to the latter. This interaction between pushforward measures and Radon--Nikodym derivatives appears in classical branches of mathematics, such as probability theory, where the pushforward measure moves the underlying probability to the probability distribution and the Radon--Nikodym derivative gives the density function; however, the interaction had not been previously specified in fractional calculus.
\end{remark}

\end{document}